\documentclass[eqthmnum]{jt-calcs}
\usepackage{amsmath}
\usepackage{mathrsfs}
\usepackage{tikz-cd}
\usepackage{varwidth}

\usepackage[backend=bibtex,style=alphabetic,sorting=nyt]{biblatex}
\bibliography{bibliography}

\newcommand{\bl}[1]{\ensuremath{\mathscr{#1}}}

\DeclareMathOperator{\sss}{ss}
\DeclareMathOperator{\ind}{ind}
\DeclareMathOperator{\IC}{IC}
\DeclareMathOperator{\std}{std}
\DeclareMathOperator{\Cl}{Cl}
\DeclareMathOperator{\Fix}{Fix}

\crefname{enumi}{}{}
\crefformat{enumi}{#2#1#3}

\newcommand{\Ql}{\ensuremath{\overline{\mathbb{Q}}_{\ell}}}

\newcommand{\Centu}[1]{\Cent_{\mathrm{U}}(#1)}
\renewcommand{\epsilon}{\varepsilon}

\definecolor{asmpgray}{gray}{0.85}
\newsavebox{\asmpbox}
\newenvironment{assumption}{
\begin{lrbox}{\asmpbox}
\begin{varwidth}{15cm}
\ignorespaces
}
{
\end{varwidth}
\end{lrbox}
\begin{center}
\setlength{\fboxsep}{5pt}
\fcolorbox{black}{asmpgray}{\usebox{\asmpbox}}
\end{center}
\ignorespacesafterend
}

\title{Evaluating Characteristic Functions of Character Sheaves at Unipotent Elements}
\author{Jay Taylor}
\address{FB Mathematik, TU Kaiserslautern, Postfach 3049, 67653 Kaiserslautern, Germany}
\email{taylor@mathematik.uni-kl.de}
\mscno{2010}{20C33}{20G40}
\keywords{Finite reductive groups, character sheaves, unipotent elements}

\begin{document}
\begin{abstract}
Assume $\bG$ is a connected reductive algebraic group defined over an algebraic closure $\mathbb{K} = \overline{\mathbb{F}}_p$ of the finite field of prime order $p>0$. Furthermore, assume that $F : \bG \to \bG$ is a Frobenius endomorphism of $\bG$. In this article we give a formula for the value of any $F$-stable character sheaf of $\bG$ at a unipotent element. This formula is expressed in terms of class functions of $\bG^F$ which are supported on a single unipotent class of $\bG$. In general these functions are not determined, however we give an expression for these functions under the assumption that $Z(\bG)$ is connected, $\bG/Z(\bG)$ is simple and $p$ is a good prime for $\bG$. In this case our formula is completely explicit.
\end{abstract}

\section{Introduction}
\begin{pa}
Throughout this article $\bG$ will denote a connected reductive algebraic group defined over an algebraic closure $\mathbb{K} = \overline{\mathbb{F}}_p$ of the finite field of prime order $p>0$. Furthermore $F : \bG \to \bG$ will denote a Frobenius endomorphism defining an $\mathbb{F}_q$-rational structure $G := \bG^F$. If $H$ is a finite group then we will denote by $\Cent(H)$ the space of all class functions $f : H \to \Ql$ where $\ell \neq p$ is a prime and $\Ql$ is an algebraic closure of the field of $\ell$-adic numbers.
\end{pa}

\begin{pa}
In \cite{lusztig:1986:character-sheaves-V} Lusztig has defined a set $\widehat{\bG}$ of $\bG$-equivariant $\Ql$-perverse sheaves on $\bG$ known as character sheaves. These geometric objects have a conjectural relationship to the irreducible characters of $G$. Through this relationship one sees that the character theory of $G$ is intimately related to the geometry of $\bG$. Let $A \in \widehat{\bG}$ be an $F$-stable character sheaf, i.e.\ there exists an isomorphism $\phi_A : F^*A \to A$, then Lusztig has associated to $A$ and $\phi_A$ a class function $\chi_{A,\phi_A} \in \Cent(G)$ whose definition depends heavily upon the choice of the isomorphism $\phi_A$. In this article we consider the following problem.
\end{pa}

\begin{prob}\label{prob:evaluate-chi-A-g}
Given $g \in G$ can we explicitly determine $\chi_{A,\phi_A}(g)$?
\end{prob}

\begin{pa}
For every character sheaf $A \in \widehat{\bG}$ we assume fixed an appropriately chosen isomorphism $\phi_A : F^*A \to A$, when one exists, such that the resulting set of functions $\mathcal{B} = \{\chi_{A,\phi_A}\}$ is an orthonormal basis of $\Cent(G)$. The problem considered here is motivated by Lusztig's conjecture, mentioned above, which is stated in \cite{lusztig:1984:characters-of-reductive-groups}. Specifically, Lusztig conjectures an explicit description for the change of basis matrix between $\mathcal{B}$ and $\Irr(G)$ (the set of all irreducible characters of $G$). Thus, having a solution to \cref{prob:evaluate-chi-A-g} together with a solution to Lusztig's conjecture would provide us with a method for determining the values of the irreducible characters of $G$. In fact, to prove Lusztig's conjecture one already requires a detailed understanding of the values of the functions in $\mathcal{B}$ (see for instance \cite{lusztig:1986:on-the-character-values,shoji:1997:unipotent-characters-of-finite-classical-groups}).
\end{pa}

\begin{pa}
In this article we will consider \cref{prob:evaluate-chi-A-g} when the element $g \in G$ is unipotent. Let us denote by $\mathcal{N}_{\bG}$ the set of all pairs $(\mathcal{O},\mathscr{E})$ where $\mathcal{O}$ is a unipotent conjugacy class of $\bG$ and $\mathscr{E}$ is a $\bG$-equivariant local system on $\mathcal{O}$. This set is partitioned into subsets $\mathscr{I}[\bL,\nu]$ called \emph{blocks} where $\bL \leqslant \bG$ is a Levi complement of a parabolic subgroup of $\bG$ and $\nu \in \mathcal{N}_{\bL}$ is a cuspidal pair. The map $(\mathcal{O},\mathscr{E}) \mapsto (F^{-1}(\mathcal{O}),F^*\mathscr{E})$ defines a natural action of $F$ on the set $\mathcal{N}_{\bG}$, compatible with the blocks, and we denote by $\mathcal{N}_{\bG}^F$ the elements fixed under this action. If $\iota \in \mathscr{I}[\bL,\nu] \cap \mathcal{N}_{\bG}^F$ then we may, and will, assume that $F(\bL) = \bL$ and $\nu \in \mathcal{N}_{\bL}^F$.

For any $\iota \in \mathscr{I}[\bL,\nu]^F$ we have a corresponding irreducible character $E_{\iota} \in \Irr(W_{\bG}(\bL))$ where $W_{\bG}(\bL) = N_{\bG}(\bL)/\bL$ is the relative Weyl group of $\bL$. Note that $F$ induces an automorphism of $W_{\bG}(\bL)$, hence an action on the set $\Irr(W_{\bG}(\bL))$. This correspondence gives a bijection $\mathscr{I}[\bL,\nu]^F \to \Irr(W_{\bG}(\bL))^F$. Furthermore, if $\iota = (\mathcal{O},\mathscr{E})$ then we have a corresponding class function $Y_{\iota} \in \Cent(G)$ which is supported only on $\mathcal{O}^F$.

Assume now that $A \in \widehat{\bG}$ is an $F$-stable character sheaf whose support contains unipotent elements (c.f.\ \cref{pa:conventions}). Then there exists an $F$-stable Levi subgroup $\bL$, a cuspidal pair $\nu = (\mathcal{O}_0,\mathscr{E}_0) \in \mathcal{N}_{\bL}^F$ and a tame $F$-stable local system $\mathscr{L}$ on $Z^{\circ}(\bL)$ such that $A$ is a summand of the induced complex $\ind_{\bL}^{\bG}(A_{\mathscr{L}})$ where
\begin{equation*}
A_{\mathscr{L}} = \IC(\overline{\mathcal{O}_0}Z^{\circ}(\bL),\mathscr{E}_0\boxtimes\mathscr{L})[\dim\mathcal{O}_0 + \dim Z^{\circ}(\bL)] \in \widehat{\bL}^F.
\end{equation*}
The summands of the induced complex $\ind_{\bL}^{\bG}(A_{\mathscr{L}})$ are in bijective correspondence with the set of irreducible characters $\Irr(W_{\bG}(\bL,\mathscr{L}))$ where $W_{\bG}(\bL,\mathscr{L}) \leqslant W_{\bG}(\bL)$ is the stabiliser of $\mathscr{L}$. If $A \in \widehat{\bG}$ corresponds to $E \in \Irr(W_{\bG}(\bL,\mathscr{L}))$ then we have $A$ is $F$-stable if and only if $E$ is fixed by $F$. With this we may state our main result (which is stated precisely in \cref{thm:A}).
\end{pa}

\begin{thm*}
For any unipotent element $u \in G$ we have
\begin{equation*}
\chi_{A,\phi_A}(u) = \sum_{\iota',\iota \in \mathscr{I}[\bL,\nu]^F}\langle \widetilde{E}_{\iota},\Ind_{W_{\bG}(\bL,\mathscr{L}).F}^{W_{\bG}(\bL).F}(\widetilde{E})\rangle_{W_{\bG}(\bL).F}(-1)^{a_{\iota}}q^{(\dim\bG+a_{\iota})/2}P_{\iota',\iota}Y_{\iota'}(u).
\end{equation*}
Here $a_{\iota} \in \mathbb{Z}$ is an integer associated to $\iota$ (c.f.\ \cref{pa:a-and-b-values}) and $\widetilde{E}_{\iota}$ is the restriction to the coset $W_{\bG}(\bL).F$ of an extension of $E_{\iota}$ to $W_{\bG}(\bL) \rtimes \langle F\rangle$ (similarly for $\widetilde{E}$). These extensions are determined by (and determine) the choice of isomorphisms defining the functions $\chi_{A,\phi_A}$ and $Y_{\iota'}$. Furthermore $(P_{\iota',\iota})$ is a block of a matrix which is computable by a general algorithm (see \cite[Theorem 24.4]{lusztig:1986:character-sheaves-V}).
\end{thm*}

\begin{pa}
If the support of the character sheaf $A$ does not contain unipotent elements then $\chi_{A,\phi_A}(u) = 0$ for all unipotent elements $u \in G$, hence our result covers all $F$-stable character sheaves. The functions $Y_{\iota'}$ are not quite the functions considered in \cite[\S24]{lusztig:1986:character-sheaves-V}. Here we must scale the functions in \cite{lusztig:1986:character-sheaves-V} by a factor coming from a certain linear character introduced by Bonnaf\'e in \cite{bonnafe:2004:actions-of-rel-Weyl-grps-I}. This linear character is known in almost all cases, in particular in the case when $Z(\bG)$ is connected and $p$ is a good prime for $\bG$ (see \cref{pa:connected-centre}).
\end{pa}

\begin{pa}
All terms in the above formula are explicitly computable, except the functions $Y_{\iota'}$. In \cref{thm:Y-functions} we give an explicit formula for these functions whenever $Z(\bG)$ is connected, $\bG/Z(\bG)$ is simple and $p$ is a good prime. Thus under these assumptions our formula is completely explicit and computable. Note that this does not exhaust all cases where these functions are known to be computable. For instance, using Shoji's results in \cite{shoji:2007:generalized-green-functions-II} one could obtain a similar statement when $\bG$ is a special orthogonal or symplectic group (if $p\neq 2$ then in this situation Bonnaf\'e's linear character is also known by work of Waldspurger - see \cite[Example 6.3]{bonnafe:2004:actions-of-rel-Weyl-grps-I} and \cite[Errata]{bonnafe:2005:actions-of-rel-Weyl-grps-II}). In the appendix we include results concerning Lusztig--MacDonald--Spaltenstein induction in finite groups extended by an automorphism, which cyclically permutes isomorphic factors. We hope this to be useful in determining the coset multiplicities occurring in the theorem.
\end{pa}

\begin{pa}
Our work here is inspired by (and is a generalisation of) that of Lusztig in \cite{lusztig:1986:on-the-character-values}. In particular, if $F$ is a split Frobenius endomorphism and $\bL$ is contained in an $F$-stable parabolic subgroup of $\bG$ then this result is due to Lusztig (see \cite[2.6(e), 3.2(a), 4.9(a)]{lusztig:1986:on-the-character-values}). It is our hope to give, in the future, an explicit solution to Lusztig's conjecture relating the characteristic functions of character sheaves to irreducible characters in the case that $Z(\bG)$ is connected, $\bG/Z(\bG)$ is simple and $p$ is a good prime for $\bG$. In particular, we would like to obtain a result similar to that of \cite[Theorem 7.2]{lusztig:1986:on-the-character-values} without restriction on the characteristic. Combining such a result with \cref{thm:A,thm:Y-functions} would then give a wholly explicit way to determine the values of irreducible characters at unipotent elements for such groups. This article may be seen as a first step towards that goal.
\end{pa}

\begin{rem}
Much of this work originally appeared in the preprint \cite{taylor:2013:multiplicities-in-GGGRs} (for instance see \cite[8.6]{taylor:2013:multiplicities-in-GGGRs}). In \cite{taylor:2013:multiplicities-in-GGGRs} it is assumed throughout that $Z(\bG)$ is connected, $\bG/Z(\bG)$ is simple and $p$ is a good prime but it became apparent to the author that most of the arguments go through unchanged in the general case. Shortly after this Digne--Lehrer--Michel released a preprint \cite{digne-lehrer-michel:2013:on-character-sheaves-and-characters} concerning Kawanaka's generalised Gelfand--Graev representations and values of irreducible characters at unipotent elements. Upon releasing this current article it was communicated to the author by Fran\c{c}ois Digne that Digne--Lehrer--Michel had also independently obtained \cref{thm:A} (see \cite[Theorem 4.1(ii)]{digne-lehrer-michel:2013:on-character-sheaves-and-characters}).

Although we achieve the same result, the methods used are quite different. Our approach is as follows. The functions on the left and right hand sides of \cref{thm:A} are the characteristic functions of certain underlying complexes. In \cite{lusztig:1986:on-the-character-values} Lusztig constructs an isomorphism between these complexes. We then deduce \cref{thm:A} by determining explicitly what happens to the Frobenius endomorphism under this isomorphism. The approach of Digne--Lehrer--Michel involves working exclusively at the level of characteristic functions. In particular, they obtain the result by rephrasing the problem in terms of generalised Green functions using Lusztig's character formula \cite[Theorem 8.5]{lusztig:1985:character-sheaves} and the language of \cite{digne-lehrer-michel:2003:space-of-unipotently-supported}.

We end by noting that to obtain an explicit formula for the values of characteristic functions of character sheaves at unipotent elements one needs a statement like \cref{thm:Y-functions}. To obtain such a statement one must be particularly careful about how one chooses the isomorphism $\phi_A$ of an $F$-stable character sheaf $A$. The work of Bonnaf\'e \cite{bonnafe:2004:actions-of-rel-Weyl-grps-I} is also crucial in dealing with the case where the character sheaf is induced from a Levi subgroup which is not the complement of an $F$-stable parabolic subgroup.
\end{rem}

\begin{acknowledgments}
The author gratefully acknowledges financial support from ERC Advanced Grant 291512. He would like to thank Geordie Williamson and Sebastian Herpel for useful conversations and Gunter Malle for his notes on an early version of this paper. Finally, we thank the referee for their useful remarks.
\end{acknowledgments}

\section{Conventions}\label{sec:conventions}
\subsection{Perverse Sheaves}
\begin{pa}\label{pa:conventions}
Given a variety $\bX$ over $\mathbb{K}$ we denote by $\mathscr{D}\bX = \mathscr{D}_c^b(\bX,\overline{\mathbb{Q}}_{\ell})$ the bounded derived category of $\overline{\mathbb{Q}}_{\ell}$-constructible sheaves on $\bX$. Furthermore we denote by $\mathscr{M}\bX$ the full subcategory of $\mathscr{D}\bX$ whose objects are the perverse sheaves on $\bX$. Assume $\bH$ is a connected algebraic group acting on $\bX$ then we take the statement $A \in \mathscr{D}\bX$ is $\bH$-equivariant to be defined as in \cite[\S1.9]{lusztig:1985:character-sheaves}. If $\bX$ is itself a connected algebraic group then, unless otherwise explicitly stated, we take $\bX$-equivariance to be with respect to the natural conjugation action of $\bX$ on itself.

We will refer to an element $A \in \mathscr{D}\bX$ as a ``complex''. Recall that we may construct for any $i \in \mathbb{Z}$ the $i$th cohomology sheaf $\mathscr{H}^iA$ of the complex $A$, which is a $\Ql$-sheaf. Given any element $x \in \bX$ we then denote by $\mathscr{H}^i_xA$ the corresponding stalk of $\mathscr{H}^iA$. For any $A \in \mathscr{D}\bX$ we call $\supp(A) := \{x \in \bX \mid \mathscr{H}_x^iA \neq 0$ for some $i \in \mathbb{Z}\}$ the support of $A$. If $\varphi : \bX \to \bY$ is a morphism then we denote by $\varphi^* : \mathscr{D}\bY \to \mathscr{D}\bX$ the inverse image functor, $\varphi_* : \mathscr{D}\bX \to \mathscr{D}\bY$ the right derived direct image functor and $\varphi_! : \mathscr{D}\bX \to \mathscr{D}\bY$ the right derived direct image functor with compact support. If $\varphi$ is smooth with connected fibres of dimension $d$ then we denote by $\tilde{\varphi}$ the shifted inverse image $\varphi^*[d]$ (c.f.\ \cite[\S1.7]{lusztig:1985:character-sheaves}).
\end{pa}

\begin{pa}
Assume $\bX \subseteq \bY$ is a subvariety then for any $A \in \mathscr{D}\bY$ we denote by $A|_{\bX}$ the complex $i^*A \in \mathscr{D}\bX$ where $i : \bX \hookrightarrow \bY$ is the inclusion map (we call this the restriction of $A$ to $\bX$). Assume now that $\bX$ is a smooth open dense subset of its closure $\overline{\bX}$ and that $\mathscr{L}$ is a local system on $\bX$ (by which we mean a locally constant $\overline{\mathbb{Q}}_{\ell}$-constructible sheaf with finite dimensional stalks) then $\mathscr{L}[\dim\bX] \in \mathscr{M}\bX$ is a perverse sheaf on $\bX$. We denote by $\IC(\overline{\bX},\mathscr{L})[\dim\bX] \in \mathscr{M}\overline{\bX}$ the intersection cohomology complex determined by $\mathscr{L}$ which is an element of $\mathscr{M}\overline{\bX}$ extending $\mathscr{L}[\dim\bX]$, i.e.\ we have $\IC(\overline{\bX},\mathscr{L})|_{\bX} \cong \mathscr{L}$. We may freely consider this as an element of $\mathscr{M}\bY$ by extending $\IC(\overline{\bX},\mathscr{L})[\dim\bX]$ to $\bY$ by 0 on $\bY - \overline{\bX}$ and we will do so without explicit mention.
\end{pa}

\begin{pa}\label{pa:direct-image-commutative}
For convenience we recall here the following base change isomorphism (see \cite[(1.7.5)]{lusztig:1985:character-sheaves-I}). Assume $\bX$, $\bY$, $\bZ$ and $\bW$ are varieties over $\mathbb{K}$ and that we have a commutative diagram of morphisms
\begin{equation*}
\begin{tikzcd}
\bX \arrow{d}{f}\arrow{r}{\phi} & \bY \arrow{d}{g}\\
\bZ \arrow{r}{\psi} & \bW
\end{tikzcd}
\end{equation*}
such that $\phi$ and $\psi$ are smooth with connected fibres of common dimension. Then we have an isomorphism $f_!\circ\tilde{\phi} = \tilde{\psi}\circ g_!$ of functors $\mathscr{D}\bY \to \mathscr{D}\bZ$. In particular if every fibre of $\phi$ and $\psi$ is simply a point then we have $f_!\circ\phi^* = \psi^*\circ g_!$.
\end{pa}

\subsection{Finite Groups}
\begin{pa}\label{pa:conventions-finite-groups}
Assume $\mathcal{A}$ is a $\Ql$-algebra then the statement ``$M$ is an $\mathcal{A}$-module'' will mean that $M$ is a finite dimensional left $\mathcal{A}$-module. For any two $\mathcal{A}$-modules $M$ and $M'$ we will denote by $\Hom_{\mathcal{A}}(M,M')$ the space of all $\mathcal{A}$-module homomorphisms $f : M \to M'$. We will denote by $\Irr(\mathcal{A})$ a set of representatives for the isomorphism classes of simple $\mathcal{A}$-modules. If $G$ is a finite group and $\mathcal{A}$ is the group algebra $\Ql G$ then $\mathcal{A}$-modules will be assumed to be either left or right modules, as appropriate, and we will write $\Hom_G(M,M')$ for $\Hom_{\mathcal{A}}(M,M')$. Similarly we will write $\Irr(G)$ for $\Irr(\mathcal{A})$ which we will also identify with the set of irreducible characters determined by the simple modules.
\end{pa}

\begin{pa}\label{pa:semidirect-products}
Assume now that $\phi : G \to G$ is an automorphism and let us denote by $\widetilde{G}$ the semidirect product $G \rtimes \langle\phi\rangle$ where $\langle\phi\rangle \leqslant \Aut(G)$ is the cyclic subgroup generated by the automorphism $\phi$. Let $H \leqslant G$ be a subgroup and let $g \in G$ be such that $\phi(gHg^{-1}) = H$ then $H$ is a normal subgroup of $\widetilde{H} = H\langle(\phi(g),\phi)\rangle \leqslant \widetilde{G}$ and we denote by $H.\phi g$ the set $\{(h\phi(g),\phi) \in \widetilde{G} \mid h \in H\}$ (note that we identify $H$ with its image $H\times \{1\}$ in $\widetilde{G}$). Furthermore we denote by $\Cent(H.\phi g)$ the $\Ql$-vector space of functions $f : H.\phi g \to \Ql$ which are invariant under conjugation by $H$. We can define on $\Cent(H.\phi g)$ an inner product $\langle -,- \rangle_{H.\phi g} : \Cent(H.\phi g) \times \Cent(H.\phi g) \to \Ql$ by setting
\begin{equation*}
\langle f, f' \rangle_{H.\phi g} = \frac{1}{|H|}\sum_{h \in H}f(h\phi(g),\phi)\overline{f'(h\phi(g),\phi)},
\end{equation*}
where $\overline{\phantom{x}} : \Ql \to \Ql$ is a fixed automorphism such that $\overline{\omega} = \omega^{-1}$ for every root of unity $\omega \in \Ql^{\times}$. We may also define a $\Ql$-linear map $\Ind_{H.\phi g}^{G.\phi} : \Cent(H.\phi g) \to \Cent(G.\phi)$ by setting
\begin{equation*}
(\Ind_{H.\phi g}^{G.\phi}f)(x,\phi) = \frac{1}{|H|}\sum_{\substack{y \in G\\ (y^{-1}x\phi(y),\phi) \in H.\phi g}} f(y^{-1}x\phi(y),\phi).
\end{equation*}
\end{pa}

\begin{pa}\label{pa:conventions-coset-identification}
For any $x \in G$ we will denote by $\ad x : G \to G$ the automorphism given by $\ad x(h) = xhx^{-1}$ for all $h \in G$. As the composition $\phi\ad g$ is an automorphism of $G$ we may form, as before, the semidirect product $G \rtimes\langle\phi\ad g\rangle$. The restriction of $\phi\ad g$ to $H$ is also an automorphism of $H$ so we have $H \rtimes \langle \phi\ad g \rangle$ is naturally a subgroup of $G \rtimes \langle \phi\ad g \rangle$. We may now define a surjective homomorphism of groups $\psi_g : G \rtimes \langle \phi\ad g\rangle \to \widetilde{G}$ given by
\begin{equation*}
\psi_g(h,\phi^i\ad \phi^{1-i}(g)\cdots\phi^{-1}(g)g) = (h\phi(g)\phi^2(g)\cdots\phi^i(g),\phi^i)
\end{equation*}
for any $i \in \mathbb{N}$. The restriction of $\psi_g$ defines a bijection $G.\phi\ad g \to G.\phi$ (resp.\ $H.\phi\ad g \to H.\phi g$) which respects the action of $G$ (resp.\ $H$) by conjugation. In particular $\psi_g$ induces isometries $\Cent(G.\phi) \to \Cent(G.\phi\ad g)$ and $\Cent(H.\phi g) \to \Cent(H.\phi\ad g)$ and carries the induction map $\Ind_{H.\phi g}^{G.\phi}$ to $\Ind_{H.\phi\ad g}^{G.\phi\ad g}$.
\end{pa}


\section{Character Sheaves}\label{sec:char-sheaves}
\begin{assumption}
Recall our assumption that $\bG$ is any connected reductive algebraic group defined over an algebraic closure $\mathbb{K} = \overline{\mathbb{F}_p}$ of the finite field of prime characteristic $p>0$. Also $F : \bG \to \bG$ is a Frobenius endomorphism of $\bG$.
\end{assumption}

\begin{pa}\label{pa:grothendieck-group}
We assume fixed an $F$-stable Borel subgroup $\bB_0 \leqslant \bG$ and maximal torus $\bT_0 \leqslant \bB_0$ (the assumption of $F$-stability will not be needed until \cref{sec:rational-structures}). Let $\Phi_{\bG}$ be the roots of $\bG$ with respect to $\bT_0$ then we will denote by $\Delta_{\bG} \subseteq \Phi_{\bG}^+ \subseteq \Phi_{\bG}$ the set of simple and positive roots determined by $\bT_0 \leqslant \bB_0$. We will denote by $(W_{\bG},\mathbb{S})$ the Coxeter system of $\bG$ where $W_{\bG} = W_{\bG}(\bT_0) = N_{\bG}(\bT_0)/\bT_0$ is the Weyl group with respect to $\bT_0$ and $\mathbb{S} = \{s_{\alpha} \mid \alpha \in \Delta_{\bG}\}$ is the set of reflections determined by $\Delta_{\bG}$.

If $\bT$ is a torus then we denote by $\mathcal{S}(\bT)$ the set of (isomorphism classes) of rank 1 local systems $\mathscr{L}$ on $\bT$ such that $\mathscr{L}^{\otimes m}$ is isomorphic to $\overline{\mathbb{Q}}_{\ell}^{\times}$ for some $m$ coprime to $p$ (we call such a local system \emph{tame}). The Weyl group $W_{\bG}$ acts naturally on $\bT_0$ and this in turn gives us an action on $\mathcal{S}(\bT_0)$ by $\mathscr{L} \mapsto (w^{-1})^*\mathscr{L}$; we denote the corresponding set of orbits by $\mathcal{S}(\bT_0)/W_{\bG}$. Associated to each such local system $\mathscr{L}$ Lusztig has defined a set of character sheaves $\widehat{\bG}_{\mathscr{L}}$ (see \cite[Definition 2.10]{lusztig:1985:character-sheaves}), which depends only on the $W_{\bG}$-orbit of $\mathscr{L}$. The set of character sheaves on $\bG$ is then defined to be
\begin{equation*}
\widehat{\bG} = \bigsqcup_{\mathscr{L} \in \mathcal{S}(\bT_0)/W_{\bG}} \widehat{\bG}_{\mathscr{L}},
\end{equation*}
whose elements are irreducible $\bG$-equivariant objects in $\mathscr{M}\bG$. We will denote by $\mathscr{K}_0(\bG)$ the subgroup of the Grothendieck group of $\mathscr{M}\bG$ spanned by the character sheaves of $\bG$. We then define a bilinear form $(-:-) : (\mathscr{K}_0(\bG) \otimes \Ql) \times (\mathscr{K}_0(\bG) \otimes \Ql) \to \Ql$ by setting
\begin{equation*}
(A:A') = \begin{cases}
1 &\text{if }A \cong A',\\
0 &\text{otherwise}.
\end{cases}
\end{equation*}
for all $A$, $A' \in \widehat{\bG}$.
\end{pa}

\begin{pa}\label{pa:definition-of-induction}
We will denote by $\mathcal{Z}$ the set of all pairs $(\bL,\bP)$ such that $\bP$ is a parabolic subgroup of $\bG$ and $\bL \leqslant \bP$ is a Levi complement of $\bP$. We then define $\mathcal{Z}_{\std}$ to be the subset consisting of all standard pairs $(\bL,\bP)$, i.e.\ $\bP$ contains $\bB_0$ and $\bL$ is the unique Levi complement of $\bP$ containing $\bT_0$ (note that every pair in $\mathcal{Z}$ is conjugate to a pair in $\mathcal{Z}_{\std}$). By projecting onto the first factor of $\mathcal{Z}$ (resp.\ $\mathcal{Z}_{\std}$) we obtain the set of Levi (resp.\ standard Levi) subgroups $\mathcal{L}$ (resp.\ $\mathcal{L}_{\std}$).

Assume now that $(\bL,\bP) \in \mathcal{Z}$ and let $A_0 \in \mathscr{M}\bL$ be an $\bL$-equivariant perverse sheaf on $\bL$. In \cite[\S4.1]{lusztig:1985:character-sheaves} Lusztig has associated to $A_0$ a complex $\ind_{\bL\subseteq\bP}^{\bG}(A_0) \in\mathscr{D}\bG$, which we call the induced complex. We recall the construction of this complex following \cite[\S4.1]{lusztig:1985:character-sheaves}. Consider the following diagram
\begin{equation}
\begin{tikzcd}
\bL & \hat{X} \arrow{l}[swap]{\pi}\arrow{r}{\sigma} & \tilde{X} \arrow{r}{\tau} & \bG
\end{tikzcd}
\end{equation}
where we have
\begin{gather*}
\begin{aligned}
\hat{X} &= \{(g,h) \in \bG \times \bG \mid h^{-1}gh \in \bP\} \quad&\quad \tilde{X} &= \{(g,h\bP) \in \bG \times (\bG/\bP) \mid h^{-1}gh \in \bP\}
\end{aligned}\\
\begin{aligned}
\pi(g,h) &= \hat{\pi}_{\bP}(h^{-1}gh) \quad&\quad \sigma(g,h) &= (g,h\bP) \quad&\quad \tau(g,h\bP) &=g
\end{aligned}
\end{gather*}
where $\hat{\pi}_{\bP} : \bP \to \bL$ is the canonical projection map. Since $A_0$ is $\bL$-equivariant there exists a canonical perverse sheaf $D$ on $\tilde{X}$ such that $\tilde{\pi}A_0 = \tilde{\sigma}D$ (c.f.\ \cref{pa:conventions}). We then define $\ind_{\bL\subseteq\bP}^{\bG}(A_0) = \tau_!D$.
\end{pa}

\begin{pa}\label{pa:induction-char-sheaves}
We say a character sheaf $A \in \widehat{\bG}$ is non-cuspidal if there exists a pair $(\bL,\bP) \in \mathcal{Z}$ (with $\bP \neq \bG$) and a character sheaf $A_0 \in \widehat{\bL}$ such that $A$ is a direct summand of the induced complex $\ind_{\bL\subseteq\bP}^{\bG}(A_0)$; otherwise we say $A$ is cuspidal. Lusztig has shown that if $A_0 \in \widehat{\bL}$ then $\ind_{\bL\subseteq\bP}^{\bG}(A_0)$ is semisimple and contained in $\mathscr{M}\bG$ (see \cite[Proposition 4.8(b)]{lusztig:1985:character-sheaves}). Furthermore for any $A \in \widehat{\bG}$ there exists a pair $(\bL,\bP) \in \mathcal{Z}$ and a cuspidal character sheaf $A_0 \in \widehat{\bL}$ such that $A$ occurs as a direct summand of $\ind_{\bL\subseteq\bP}^{\bG}(A_0)$ (see \cite[Theorem 4.4]{lusztig:1985:character-sheaves}).

Let us now fix a pair $(\bL,\bP) \in \mathcal{Z}$ such that there exists a cuspidal character sheaf $A_0 \in \widehat{\bL}$. By \cite[Proposition 3.12]{lusztig:1985:character-sheaves} we have $A_0$ is isomorphic to an intersection cohomology complex $\IC(\overline{\Sigma},\mathscr{E})[\dim\Sigma]$ where $\Sigma \subset \bL$ is the inverse image under $\bL \mapsto \bL/Z^{\circ}(\bL)$ of an isolated conjugacy class and $\mathscr{E}$ is a local system on $\Sigma$. From the proof of this result we know that $(\Sigma,\mathscr{E})$ is a cuspidal pair in the sense of \cite[Definition 2.4]{lusztig:1984:intersection-cohomology-complexes}. Of particular interest to us will be the special case where $\supp(A_0) \cap \bL_{\uni} \neq \emptyset$, where for any connected reductive algebraic group $\bH$ we denote by $\bH_{\uni}$ the variety of unipotent elements. Assume this is so then there exists a triple $(\mathcal{O}_0,\mathscr{E}_0,\mathscr{L})$ consisiting of: a unipotent conjugacy class $\mathcal{O}_0 \subset \bL$, a cuspidal local system $\mathscr{E}_0$ on $\mathcal{O}_0$ and a local system $\mathscr{L} \in \mathcal{S}(Z^{\circ}(\bL))$ such that $A_0$ is isomorphic to 
\begin{equation}\label{eq:cusp-uni-sup}
A_{\mathscr{L}} := \IC(\overline{\Sigma},\mathscr{E}_0\boxtimes\mathscr{L})[\dim\Sigma]
\end{equation}
where $\Sigma = \mathcal{O}_0Z^{\circ}(\bL)$. Here we consider $\mathscr{E}_0\boxtimes\mathscr{L}$ as a local system on $\mathcal{O}_0 \times Z^{\circ}(\bL)$ which we identify with the open subset $\mathcal{O}_0Z^{\circ}(\bL) \subseteq \overline{\mathcal{O}_0}Z^{\circ}(\bL)$ under the multiplication morphism in $\bL$. With all of this we may now define a map
\begin{equation}\label{eq:ind-P-construction}
(\bL,\bP,\mathcal{O}_0,\mathscr{E}_0,\mathscr{L}) \longrightarrow \ind_{\bL\subseteq\bP}^{\bG}(A_{\mathscr{L}})
\end{equation}
\end{pa}

\begin{assumption}
From this point forward $\Sigma$ will always denote the variety $\mathcal{O}_0Z^{\circ}(\bL)$ where $\mathcal{O}_0 \subset \bL$ is a unipotent conjugacy class supporting a cuspidal local system.
\end{assumption}

\begin{pa}
Let $A \in \widehat{\bG}$ be a character sheaf which occurs as a direct summand of the induced complex $\ind_{\bL\subseteq\bP}^{\bG}(A_0)$ (where $A_0 \in \widehat{\bL}$ is still assumed to be cuspidal) then by \cite[\S2.9]{lusztig:1986:on-the-character-values} we have
\begin{equation}\label{eq:supp}
\supp A = \bigcup_{x\in\bG} x(\supp A_0)\bU_{\bP}x^{-1}
\end{equation}
where $\bU_{\bP}\leqslant \bP$ is the unipotent radical of $\bP$. In particular we have $\supp A \cap \bG_{\uni} \neq \emptyset$ if and only if $\supp A_0 \cap \bL_{\uni} \neq \emptyset$. Hence if we are only interested in character sheaves whose support contains unipotent elements then we need only concern ourselves with those character sheaves occurring in an induced complex of the form given in \cref{eq:cusp-uni-sup}. We end our discussion of induced complexes with the following lemma which gives some facts concerning the inverse image of an induced complex.
\end{pa}

\begin{lem}\label{lem:F-action-ind}
We denote by $\bH$ a connected reductive algebraic group and by $\bP$ a parabolic subgroup of $\bH$ with Levi complement $\bL$. Furthermore we assume that $\bG$ is a connected reductive algebraic group and $i : \bG \to \bH$ is a bijective morphism of varieties.
\begin{enumerate}[label=(\alph*)]
\item For any $A_0 \in \widehat{\bL}$ we have
\begin{equation*}
i^*\ind_{\bL\subseteq\bP}^{\bH}(A_0) = \ind_{i^{-1}(\bL)\subseteq i^{-1}(\bP)}^{\bG}(i^*A_0).
\end{equation*}
Furthermore any character sheaf $A \in \widehat{\bH}$ is cuspidal if and only if $i^*A \in \widehat{\bG}$ is cuspidal.

\item Let us also assume that $\bG = \bH$ and $\bL = i(\bL)$ then any isomorphism $\phi : i^*A_0 \to A_0$ induces an isomorphism
\begin{equation*}
\tilde{\phi} : i^*\ind_{\bL\subseteq i(\bP)}^{\bH}(A_0) \to \ind_{\bL\subseteq \bP}^{\bH}(A_0).
\end{equation*}
\end{enumerate}
\end{lem}

\begin{proof}
Firstly, by \cite[\S24.1 - Proposition B]{humphreys:1975:linear-algebraic-groups}, we have $i$ induces a bijection between the Borel subgroups of $\bG$ and $\bH$ so this implies that $\bQ := i^{-1}(\bP)$ is a parabolic subgroup of $\bG$ with Levi complement $\bM := i^{-1}(\bL)$. For any object $\square$ introduced in \cref{pa:definition-of-induction} we write $\square_{\bL\subseteq\bP}^{\bH}$ (resp.\ $\square_{\bM\subseteq\bQ}^{\bG}$) to indicate that it is defined with respect to $\ind_{\bL\subseteq\bP}^{\bH}$ (resp.\ $\ind_{\bM\subseteq\bQ}^{\bG}$). We now have a diagram
\begin{center}
\begin{tikzcd}
\bM \arrow{d}[swap]{i} & \hat{X}_{\bM\subseteq\bQ}^{\bG} \arrow{l}[swap]{\pi_{\bM\subseteq\bQ}^{\bG}}\arrow{d}[swap]{i}\arrow{r}{\sigma_{\bM\subseteq\bQ}^{\bG}} & \tilde{X}_{\bM\subseteq\bQ}^{\bG} \arrow{d}[swap]{i}\arrow{r}{\tau_{\bM\subseteq\bQ}^{\bG}} & \bG \arrow{d}[swap]{i}\\
\bL & \hat{X}_{\bL\subseteq\bP}^{\bH} \arrow{l}[swap]{\pi_{\bL\subseteq\bP}^{\bH}} \arrow{r}{\sigma_{\bL\subseteq\bP}^{\bH}} & \tilde{X}_{\bL\subseteq\bP}^{\bH} \arrow{r}{\tau_{\bL\subseteq\bP}^{\bH}} & \bH
\end{tikzcd}
\end{center}
where the vertical maps are the obvious actions of $i$. The squares of the above diagram are clearly commutative. Let $D$ be the canonical perverse sheaf on $\tilde{X}_{\bL\subseteq\bP}^{\bH}$ satisfying $\tilde{\pi}_{\bL\subseteq\bP}^{\bH}A_0 = \tilde{\sigma}_{\bL\subseteq\bP}^{\bH}D$. The fibres of $\pi_{\bL\subseteq\bP}^{\bH}$ and $\pi_{\bM\subseteq\bQ}^{\bG}$ have the same dimension, as do the fibres of $\sigma_{\bL\subseteq\bP}^{\bH}$ and $\sigma_{\bM\subseteq\bQ}^{\bG}$, therefore we have $i^*D$ satisfies $\tilde{\pi}_{\bM\subseteq\bQ}^{\bG}i^*A_0 = \tilde{\sigma}_{\bM\subseteq\bQ}^{\bG} i^*D$ because the inverse image is contravariant. By definition we have $\ind_{\bL\subseteq\bP}^{\bH}(A_0) = (\tau_{\bL\subseteq\bP}^{\bH})_!D$ and $\ind_{\bM\subseteq\bQ}^{\bG}(i^*A_0) = (\tau_{\bM\subseteq\bQ}^{\bG})_!i^*D$ so the first part of the lemma follows if we can show the equality $i^*(\tau_{\bL\subseteq\bP}^{\bH})_!D = (\tau_{\bM\subseteq\bQ}^{\bG})_!i^*D$ but this is just \cref{pa:direct-image-commutative}. The conclusion concerning cuspidality is an immediate consequence of the first part, which proves (a).

We now prove (b). Let $\bQ = i(\bP)$ then collapsing the notation $\square_{\bL\subset\bP}^{\bH}$ simply to $\square_{\bP}$ we have a diagram
\begin{center}
\begin{tikzcd}
\bL \arrow{d}[swap]{i} & \hat{X}_{\bP} \arrow{l}[swap]{\pi_{\bP}}\arrow{d}[swap]{i}\arrow{r}{\sigma_{\bP}} & \tilde{X}_{\bP} \arrow{d}[swap]{i}\arrow{r}{\tau_{\bP}} & \bH \arrow{d}[swap]{i}\\
\bL & \hat{X}_{\bQ} \arrow{l}[swap]{\pi_{\bQ}} \arrow{r}{\sigma_{\bQ}} & \tilde{X}_{\bQ} \arrow{r}{\tau_{\bQ}} & \bH
\end{tikzcd}
\end{center}
with commutative squares. Let $K$ (resp.\ $K'$) be the canonical perverse sheaf on $\tilde{X}_{\bP}$ (resp.\ $\tilde{X}_{\bQ}$) satisfying $\tilde{\pi}_{\bP}A_0 = \tilde{\sigma}_{\bP}K$ (resp.\ $\tilde{\pi}_{\bQ}A_0 = \tilde{\sigma}_{\bQ}K'$). Now $\tilde{\pi}_{\bP}\phi$ defines an isomorphism $\tilde{\pi}_{\bP}i^*A_0 \to \tilde{\pi}_{\bP}A_0$ and using the commutativity of the above diagram we may view this as an isomorphism $\tilde{\sigma}_{\bP}i^*K' \to \tilde{\sigma}_{\bP}K$. As $\sigma_{\bP}$ is smooth with connected fibres we have $\tilde{\sigma}_{\bP}$ is a fully faithful functor (see \cite[1.8.3]{lusztig:1985:character-sheaves}), hence there exists a unique isomorphism $\phi' : i^*K' \to K$ such that $\tilde{\pi}_{\bP}\phi = \tilde{\sigma}_{\bP}\phi'$. Using the arguments above we see that $\tilde{\phi} = (\tau_{\bP})_!\phi'$ gives the required isomorphism.
\end{proof}

\section{The Space of Unipotently Supported Class Functions}\label{sec:space-unip-supp-class-func}
\begin{pa}\label{pa:gen-spring-cor}
Let $\mathcal{N}_{\bG}$ denote the set of all pairs $\iota = (\mathcal{O},\mathscr{E})$ where $\mathcal{O}$ is a unipotent conjugacy class of $\bG$ and $\mathscr{E}$ is an irreducible $\bG$-equivariant local system on $\mathcal{O}$. We denote by $\mathcal{N}_{\bG}^0 \subseteq \mathcal{N}_{\bG}$ the subset consisting of those pairs $(\mathcal{O},\mathscr{E})$ such that $\mathscr{E}$ is a cuspidal local system on $\mathcal{O}$ (see \cite[Definition 2.4]{lusztig:1984:intersection-cohomology-complexes}); we call the elements of $\mathcal{N}_{\bG}^0$ cuspidal pairs. Given a pair $\iota$ we will denote the class $\mathcal{O}$ by $\mathcal{O}_{\iota}$ and the local system $\mathscr{E}$ by $\mathscr{E}_{\iota}$.

Let us denote by $\widetilde{\mathcal{M}}_{\bG}$ the set of all pairs $(\bL,\nu)$ consisting of a Levi subgroup $\bL \in \mathcal{L}$ and a cuspidal pair $\nu \in \mathcal{N}_{\bL}^0$. We have $\bG$ acts naturally on $\widetilde{\mathcal{M}}_{\bG}$ by conjugation and we denote by $[\bL,\nu]$ the orbit containing $(\bL,\nu)$. We also denote by $\mathcal{M}_{\bG}$ the set of all such orbits. Recall that in \cite[Theorem 6.5]{lusztig:1984:intersection-cohomology-complexes} Lusztig has associated to every pair $\iota \in \mathcal{N}_{\bG}$ a unique orbit $\mathcal{C}_{\iota} \in \mathcal{M}_{\bG}$.

\begin{assumption}
For each $\iota \in \mathcal{N}_{\bG}$ we will now choose a representative $(\bL_{\iota},\nu_{\iota}) \in \mathcal{C}_{\iota}$ with $\bL_{\iota} \in \mathcal{L}_{\std}$. For convenience, we will assume that if $\iota, \iota' \in \mathcal{N}_{\bG}$ satisfy $\mathcal{C}_{\iota} = \mathcal{C}_{\iota'}$ then $(\bL_{\iota},\nu_{\iota}) = (\bL_{\iota'},\nu_{\iota'})$.
\end{assumption}

\noindent In \cite{lusztig:1984:intersection-cohomology-complexes} it was shown that we have a disjoint union
\begin{equation*}
\mathcal{N}_{\bG} = \bigsqcup_{[\bL,\upsilon] \in \mathcal{M}_{\bG}} \bl{I}[\bL,\upsilon] \qquad\text{where}\qquad \bl{I}[\bL,\upsilon] = \{\iota \in \mathcal{N}_{\bG} \mid (\bL_{\iota},\upsilon_{\iota}) \in [\bL,\upsilon]\}.
\end{equation*}
We call $\bl{I}[\bL,\upsilon]$ a \emph{block} of $\mathcal{N}_{\bG}$. If $\nu \in \mathcal{N}_{\bL}^0$ is cuspidal then $W_{\bG}(\bL) = N_{\bG}(\bL)/\bL$ is a Coxeter group (see \cite[Theorem 9.2(a)]{lusztig:1984:intersection-cohomology-complexes}) and we have a bijection
\begin{equation}\label{eq:gen-spring-cor}
\bl{I}[\bL,\nu] \longleftrightarrow \Irr(W_{\bG}(\bL))
\end{equation}
for all $[\bL,\nu] \in \mathcal{M}_{\bG}$, which we denote $\iota \mapsto E_{\iota}$. This is known as the generalised Springer correspondence. If $\bL$ is a torus then $\nu$ is simply the pair consisting of the trivial class and the trivial local system, in which case this bijection is the classical Springer correspondence and we call $\mathscr{I}[\bL,\nu]$ the Springer block.
\end{pa}

\begin{pa}\label{pa:decomp-by-end-algebra}
To describe the correspondence in \cref{eq:gen-spring-cor} we will need to recall a description of semisimple objects in $\mathscr{M}\bG$ following \cite[3.7]{lusztig:1984:intersection-cohomology-complexes}. Let $K \in \mathscr{M}\bG$ be semisimple and let $\mathcal{A} = \End(K)$ be the endomorphism algebra of $K$, by which we mean the algebra $\Hom_{\mathscr{M}\bG}(K,K)$ in the category $\mathscr{M}\bG$. Assume $E$ is any finite dimensional $\mathcal{A}$-module then we define
\begin{equation}\label{eq:K-nu-E-L}
K_E = \Hom_{\mathcal{A}}(E,K) \in \mathscr{M}\bG.
\end{equation}
To see that this is an object of $\mathscr{M}\bG$ we can construct this in the following way. Let us pick a presentation $\mathcal{A}^m \overset{\varphi}{\to} \mathcal{A}^n \to E \to 0$ of the module $E$. Applying $\Hom_{\mathcal{A}}(-,K)$ and using the fact that we have an isomorphism $K \cong \Hom_{\mathcal{A}}(\mathcal{A},K)$ of $\mathcal{A}$-modules we get a diagram
\begin{equation}\label{eq:diag-Hom-MG}
\begin{tikzcd}
0 \arrow{r}{} & \Hom_{\mathcal{A}}(E,K) \arrow{d}{\cong}\arrow{r}{} & \Hom_{\mathcal{A}}(\mathcal{A}^n,K) \arrow{d}{\cong}\arrow{r}{\varphi^*} & \Hom_{\mathcal{A}}(\mathcal{A}^m,K) \arrow{d}{\cong}\\
0 \arrow{r}{} & \Ker_{\mathscr{M}\bG}(\varphi^*) \arrow{r}{} & K^n \arrow{r}{\varphi^*} & K^m
\end{tikzcd}
\end{equation}
with exact rows ($\varphi^*$ is the map induced by $\varphi$ and the vertical arrows are $\mathcal{A}$-module isomorphisms). The $\mathcal{A}$-module homomorphism $\varphi^*$ is also a morphism in $\mathscr{M}\bG$ because $\mathscr{M}\bG$ is $\Ql$-linear and the kernel exists because $\mathscr{M}\bG$ is an abelian category. Hence $\Hom_{\mathcal{A}}(E,K) \cong \Ker(\varphi^*) \in \mathscr{M}\bG$ as desired. Now $K_E$ is a simple object of $\mathscr{M}\bG$ if and only if $E$ is a simple $\mathcal{A}$-module and we have a canonical isomorphism
\begin{equation}\label{eq:K-nu-L-decomp}
\bigoplus_{E \in \Irr(\mathcal{A})} (E \otimes K_E) \cong K
\end{equation}
in $\mathscr{M}\bG$ given by $e\otimes f \mapsto f(e)$ in each summand.
\end{pa}

\begin{pa}\label{pa:unip-sup-char-sheaves}
Let $\bL \in \mathcal{L}$ be a Levi subgroup and let $\nu = (\mathcal{O}_0,\mathscr{E}_0) \in \mathcal{N}_{\bL}^0$ be a cuspidal pair then following \cite[2.3]{lusztig:1986:on-the-character-values} we define for any local system $\mathscr{L} \in \mathcal{S}(Z^{\circ}(\bL))$ a semisimple perverse sheaf $K_{\mathscr{L}} \in \mathscr{M}\bG$ in the following way. We define an open subset $\Sigma_{\reg}= \mathcal{O}_0\cdot Z^{\circ}(\bL)_{\reg} \subseteq \Sigma$ where $Z^{\circ}(\bL)_{\reg} = \{z \in Z^{\circ}(\bL) \mid C_{\bG}^{\circ}(z) = \bL\}$ and we denote by $Y$ the locally closed smooth irreducible subvariety of $\bG$ given by $\cup_{x \in \bG} x\Sigma_{\reg}x^{-1}$. We have the following diagram
\begin{equation}
\begin{tikzcd}
\Sigma & \hat{Y} \arrow{l}[swap]{\alpha}\arrow{r}{\beta} & \tilde{Y} \arrow{r}{\gamma} & Y
\end{tikzcd}
\end{equation}
where
\begin{gather*}
\begin{aligned}
\hat{Y} &= \{(g,x) \in \bG \times \bG \mid x^{-1}gx \in \Sigma\} \quad&\quad \tilde{Y} &= \{(g,x\bL) \in \bG \times (\bG/\bL) \mid x^{-1}gx \in \bL\}
\end{aligned}\\
\begin{aligned}
\alpha(g,x) &= x^{-1}gx \quad&\quad \beta(g,x) &= (g,x\bL) \quad&\quad \gamma(g,x\bL) &= g.
\end{aligned}
\end{gather*}
Since the local system $\mathscr{E}_0$ is $\bL$-equivariant there exists a canonical local system $\tilde{\mathscr{E}}_0$ on $\tilde{Y}$ such that $\beta^*\tilde{\mathscr{E}}_0 = \alpha^*(\mathscr{E}_0\boxtimes\Ql)$ (see \cite[1.9.3]{lusztig:1985:character-sheaves-I}). Let $\delta : \tilde{Y} \to Z^{\circ}(\bL)$ be the map given by $\delta(g,x\bL) = (xgx^{-1})_{\sss}$ (where $h_{\sss}$ is the semisimple part of $h$ for all $h \in \bG$) then $\tilde{\mathscr{L}} = \delta^*\mathscr{L}$ is a local system on $\tilde{Y}$ hence so is $\tilde{\mathscr{E}}_0\otimes\tilde{\mathscr{L}}$. By \cite[3.2]{lusztig:1984:intersection-cohomology-complexes} we have $\gamma$ is a Galois covering so $\gamma_* = \gamma_!$ because $\gamma$ is finite (hence proper), which means $\gamma_*(\tilde{\mathscr{E}}_0\otimes\tilde{\mathscr{L}}) = \gamma_!(\tilde{\mathscr{E}}_0\otimes\tilde{\mathscr{L}})$ is a semisimple local system on $Y$. We now define $K_{\mathscr{L}} \in \mathscr{M}\bG$ to be the complex $\IC(\overline{Y},\gamma_*(\tilde{\mathscr{E}}_0\otimes\tilde{\mathscr{L}}))[\dim Y]$. With this we have defined a map
\begin{equation}\label{eq:ind-w-P-construction}
(\bL,\mathcal{O}_0,\mathscr{E}_0,\mathscr{L}) \longrightarrow K_{\mathscr{L}}
\end{equation}
\end{pa}

\begin{pa}\label{pa:endomorphism-algebra-A}
Let us keep the notation of \cref{pa:unip-sup-char-sheaves}. We denote by $N_{\bG}(\bL,\mathscr{L})$ the set of all elements $n \in N_{\bG}(\bL)$ such that $(\ad n)^*\mathscr{L}$ is isomorphic to $\mathscr{L}$, where $\ad n : \bG \to \bG$ is the conjugation morphism given by $(\ad n)(g) = ngn^{-1}$ for all $g \in \bG$. Clearly $\bL$ is a normal subgroup of $N_{\bG}(\bL,\mathscr{L})$ and we denote the quotient group $N_{\bG}(\bL,\mathscr{L})/\bL$ by $W_{\bG}(\bL,\mathscr{L})$. Let us denote by $\mathcal{A}_{\mathscr{L}}$ the endomorphism algebra of the semisimple perverse sheaf $K_{\mathscr{L}}$. By \cite[\S2.4(a)]{lusztig:1986:on-the-character-values} there exists a set of basis elements $\{\Theta_v \mid v \in W_{\bG}(\bL,\mathscr{L})\} \subset \mathcal{A}_{\mathscr{L}}$ such that the map $v \mapsto \Theta_v$ defines a canonical algebra isomorphism
\begin{equation}\label{eq:end-grp-alg-iso}
\Ql W_{\bG}(\bL,\mathscr{L}) \cong \mathcal{A}_{\mathscr{L}}.
\end{equation}
Let us now assume that $\mathscr{L} = \Ql$ then we denote $\mathcal{A}_{\mathscr{L}}$ simply by $\mathcal{A}$ and $K_{\mathscr{L}}$ by $K$. In this case we have $W_{\bG}(\bL,\mathscr{L}) = W_{\bG}(\bL)$ hence $\mathcal{A}$ is isomorphic to $\Ql W_{\bG}(\bL)$. Using the description given in \cref{pa:decomp-by-end-algebra} we have $E\mapsto K_E = \Hom_{W_{\bG}(\bL)}(E,K)$ gives a bijection between simple $W_{\bG}(\bL)$-modules and the simple summands of $K$ (up to isomorphism). Furthermore by \cite[Theorem 6.5]{lusztig:1984:intersection-cohomology-complexes} we have for each $E$ that there exists a unique pair $\iota \in \mathscr{I}[\bL,\nu]$ satisfying
\begin{equation*}
K_E|_{\bG_{\uni}} \cong \IC(\overline{\mathcal{O}}_{\iota},\mathscr{E}_{\iota})[\dim\mathcal{O}_{\iota} + \dim Z^{\circ}(\bL)].
\end{equation*}
The composition of the maps $E \mapsto K_E \mapsto \iota$ then gives the bijection in \cref{eq:gen-spring-cor}. We will denote by $K_{\iota} \in \mathscr{D}\bG$ the extension by 0 of the complex $K_E|_{\bG_{\uni}}$. Clearly we have $\supp(K_{\iota})\subseteq \bG_{\uni}$.
\end{pa}

\begin{rem}
In \cite[3.4]{lusztig:1984:intersection-cohomology-complexes} the definition of $N_{\bG}(\bL,\mathscr{L})$ includes the condition that $(\ad n)(\Sigma) = \Sigma$. However this is automatically satisfied in our situation by \cite[Theorem 9.2(b)]{lusztig:1984:intersection-cohomology-complexes}.
\end{rem}

\begin{pa}\label{pa:K-nu-L-ind-A-nu-L}
We now end this section by explicitly describing how $K_{\mathscr{L}}$ is related to $\ind_{\bL\subseteq\bP}^{\bG}(A_{\mathscr{L}})$ (for this we follow \cite{lusztig:1984:intersection-cohomology-complexes}). Here we assume $(\bL,\bP) \in \mathcal{Z}$ and $\nu = (\mathcal{O}_0,\mathscr{E}_0) \in \mathcal{N}_{\bL}^0$. Recalling the notation of \cref{pa:definition-of-induction} we wish to describe the complex $D$. To do this we first observe that
\begin{equation*}
\tilde{\pi}A_{\mathscr{L}} = \IC(\hat{X}',\pi^*(\mathscr{E}_0\boxtimes\mathscr{L}))[\dim \overline{\Sigma}+\dim\bG + \dim\bU_{\bP}]
\end{equation*}
where $\bU_{\bP}$ is the unipotent radical of $\bP$ and $\hat{X}' = \pi^{-1}(\overline{\Sigma}) = \{(g,h) \in \bG \times \bG \mid h^{-1}gh \in \overline{\Sigma}\cdot\bU_{\bP}\}$. Note that for $l \in \bL$ we have $\pi^{-1}(l) = \{(g,h) \in \bG \times \bG \mid h^{-1}gh \in l\cdot \bU_{\bP}\}$ hence the fibres of $\pi$ have dimension $\dim\bG+\dim\bU_{\bP}$.

Let us fix a set of coset representatives $\{h_1,\dots,h_k\}$ of $\bP$ in $\bG$ then we define $j : \tilde{X} \to \hat{X}$ by setting $j(g,h\bP) = (g,h_i)$ where $1\leqslant i \leqslant k$ is the unique index such that $h \in h_i\bP$. By the construction in \cite[1.9.3]{lusztig:1985:character-sheaves-I} we see that
\begin{equation*}
D = j^*\IC(\hat{X}',\pi^*(\mathscr{E}_0\boxtimes\mathscr{L}))[\dim \overline{\Sigma}+\dim\bG + \dim\bU_{\bP} - \dim\bP] = \IC(\tilde{X}',\overline{\mathscr{E}_0\boxtimes\mathscr{L}})[\dim \tilde{X}']
\end{equation*}
where $\tilde{X}' = j^{-1}(\hat{X}') = \{(g,h_i\bP) \in \bG \times \bG/\bP \mid h_i^{-1}gh_i \in \overline{\Sigma}\cdot\bU_{\bP}\}$ and $\overline{\mathscr{E}_0\boxtimes\mathscr{L}} = j^*\pi^*(\mathscr{E}_0\boxtimes\mathscr{L})$. Note that the equality $\dim \tilde{X}' = \dim Y = \dim\bG - \dim\bP + \dim\overline{\Sigma}+ \dim\bU_{\bP}$ is given in the proof of \cite[Lemma 4.3(a)]{lusztig:1984:intersection-cohomology-complexes}.

Let $Y$ and $\tilde{Y}$ be as in \cref{pa:unip-sup-char-sheaves} then by \cite[Lemma 4.3(b)]{lusztig:1984:intersection-cohomology-complexes} we have $\tau(\tilde{X}') = \overline{Y}$ hence $\tau^{-1}(Y) \subset \tilde{X}' \subset \tilde{X}$. Furthermore, by \cite[Lemma 4.3(c)]{lusztig:1984:intersection-cohomology-complexes}, the map $(g,h\bL) \to (g,h\bP)$ defines an isomorphism $\kappa : \tilde{Y} \to \tau^{-1}(Y)$. As $\tau^{-1}(Y) \subset \tilde{X}'$ we have $D|_{\tau^{-1}(Y)} \cong \overline{\mathscr{E}_0\boxtimes\mathscr{L}}|_{\tau^{-1}(Y)}[\dim\tilde{X}']$ and $\kappa^*(\overline{\mathscr{E}_0\boxtimes\mathscr{L}}|_{\tau^{-1}(Y)}) \cong \tilde{\mathscr{E}}_0 \otimes \tilde{\mathscr{L}}$ (see \cite[4.4]{lusztig:1984:intersection-cohomology-complexes}) hence
\begin{equation*}
(\tau_!D)|_Y = \tau_!(D_{\tau^{-1}(Y)}) \cong \gamma_*\kappa^*(\overline{\mathscr{E}_0\boxtimes\mathscr{L}}|_{\tau^{-1}(Y)}[\dim\tilde{X}']) \cong \gamma_*(\tilde{\mathscr{E}}_0\otimes\tilde{\mathscr{L}})[\dim Y].
\end{equation*}
Here we have applied \cref{pa:direct-image-commutative} to obtain $\tau_! = \ID^*\circ\tau_! = \gamma_!\circ\kappa^* = \gamma_*\circ\kappa^*$ where the last equality follows from the fact that $\gamma$ is proper. With this we have the following result of Lusztig.
\end{pa}

\begin{prop}[{}{Lusztig, \cite[Proposition 4.5]{lusztig:1984:intersection-cohomology-complexes}}]\label{prop:K-nu-L-iso-to-induction}
The complex $\tau_!D$ is a perverse sheaf and is canonically isomorphic to $\IC(\overline{Y},\gamma_*(\tilde{\mathscr{E}}_0\otimes\tilde{\mathscr{L}}))[\dim Y]$. In particular $\ind_{\bL \subseteq\bP}^{\bG}(A_{\mathscr{L}})$ and $K_{\mathscr{L}}$ are canonically isomorphic in $\mathscr{M}\bG$.
\end{prop}

\begin{rem}
Note that \cref{prop:K-nu-L-iso-to-induction} implies that $\ind_{\bL\subseteq\bP}^{\bG}(A_{\mathscr{L}})$ does not depend upon the choice of parabolic $\bP$ containing $\bL$ so in this situation we will simply write $\ind_{\bL}^{\bG}(A_{\mathscr{L}})$ when convenient.
\end{rem}

\section{\texorpdfstring{Bases of the Endomorphism Algebra $\mathcal{A}$}{Bases of the Endomorphism Algebra A}}\label{sec:bases-of-end-A}
\begin{pa}
Assume $\bL \in \mathcal{L}$ is a Levi subgroup supporting a cuspidal pair $(\mathcal{O}_0,\mathscr{E}_0) \in \mathcal{N}_{\bL}^0$ then we denote by $K$ the image of $(\bL,\mathcal{O}_0,\mathscr{E}_0,\Ql)$ under the map in \cref{eq:ind-w-P-construction} and by $\mathcal{A}$ the endomorphism algebra of $K$. As was mentioned in \cref{eq:end-grp-alg-iso} Lusztig has defined an isomorphism between $\mathcal{A}$ and the group algebra $\Ql W_{\bG}(\bL)$ by specifying a set of basis elements $\{\Theta_v \mid v \in W_{\bG}(\bL)\} \subset \mathcal{A}$. However, in \cite[\S6.A]{bonnafe:2004:actions-of-rel-Weyl-grps-I} Bonnaf\'{e} has defined an alternative basis $\{\Theta_v' \mid v \in W_{\bG}(\bL)\} \subset \mathcal{A}$ which also defines such an isomorphism. In this article we will need to use both bases and we recall here results of Bonnaf\'{e} concerning the relationship between the two.
\end{pa}

\begin{assumption}
We now assume that $u_0 \in \mathcal{O}_0$ is a fixed unipotent element. If $\bL$ and $\mathcal{O}_0$ are $F$-stable then we assume that $u_0 \in \mathcal{O}_0^F$.
\end{assumption}

\begin{pa}\label{pa:Lusztig-iso-A}
Assume $v \in W_{\bG}(\bL)$ then by \cite[eq.\ (5.4)]{bonnafe:2004:actions-of-rel-Weyl-grps-I} we may find a representative $\dot{v} \in N_{\bG}(\bL) \cap C_{\bG}^{\circ}(u_0)$ and we will assume that all such representatives are chosen in this way. Let us recall the notation of \cref{pa:unip-sup-char-sheaves}. For any $v \in W_{\bG}(\bL)$ we have by \cite[Theorem 9.2(b)]{lusztig:1984:intersection-cohomology-complexes} that there exists an isomorphism $\theta_v : \mathscr{E}_0\boxtimes\Ql \to (\ad\dot{v})^*(\mathscr{E}_0\boxtimes\Ql)$ which we assume fixed. This isomorphism induces an isomorphism $\hat{\theta}_v : \tilde{\mathscr{E}}_0 \to \gamma_v^*\tilde{\mathscr{E}}_0$ where $\gamma_v : \widetilde{Y} \to \widetilde{Y}$ is given by $\gamma_v(g,x\bL) = (g,x\dot{v}^{-1}\bL)$ (see the proof of \cite[Proposition 3.5]{lusztig:1984:intersection-cohomology-complexes}). By definition we have $\gamma_v\gamma = \gamma$ hence $\gamma_*\hat{\theta}_v$ defines an endomorphism of $\gamma_*\tilde{\mathscr{E}}_0$. As $K$ is the intersection cohomology complex $\IC(\overline{Y},\gamma_*\tilde{\mathscr{E}}_0)$ we may then define $\Theta_v$ to be the unique endomorphism of $K$ extending $\gamma_*\hat{\theta}_v$.

By \cite[Theorem 9.2(d)]{lusztig:1984:intersection-cohomology-complexes} there is a unique isomorphism $\theta_v$ for each $v \in W_{\bG}(\bL)$ such that $\Theta_v$ induces the identity on $\mathscr{H}_u^{-\dim Y}(K)$ where $u$ is any element of the induced unipotent class $\Ind_{\bL}^{\bG}(\mathcal{O}_0)$ (see \cite[Corollary 7.3(a)]{lusztig:1984:intersection-cohomology-complexes}). With this choice we have
\begin{equation}\label{eq:algebra-iso}
v \mapsto \Theta_v
\end{equation}
defines the required algebra isomorphism $\Ql W_{\bG}(\bL) \to \mathcal{A}$. Assume now that $\mathscr{L}$ is a local system on $Z^{\circ}(\bL)$. Let $K_{\mathscr{L}}$ be the image of $(\bL,\mathcal{O}_0,\mathscr{E}_0,\mathscr{L})$ under the map in \cref{eq:ind-w-P-construction} and let $\mathcal{A}_{\mathscr{L}}$ be the endomorphism algebra of $K_{\mathscr{L}}$ then from the discussion in \cite[2.3]{lusztig:1986:on-the-character-values} we have an embedding of algebras $\mathcal{A}_{\mathscr{L}} \hookrightarrow \mathcal{A}$ which corresponds under \cref{eq:algebra-iso} to the natural embedding $\Ql W_{\bG}(\bL,\mathscr{L}) \hookrightarrow \Ql W_{\bG}(\bL)$. In particular the restriction of \cref{eq:algebra-iso} to $W_{\bG}(\bL,\mathscr{L})$ defines the isomorphism mentioned in \cref{eq:end-grp-alg-iso}.
\end{pa}

\begin{rem}
Assume $\bL$ and $\mathcal{O}_0$ are $F$-stable. It is clear to see that if $\dot{v} \in N_{\bG}(\bL)\cap C_{\bG}^{\circ}(u_0)$ is a representative for $v \in W_{\bG}(\bL)$ then $\dot{v}^{-1} \in N_{\bG}(\bL)\cap C_{\bG}^{\circ}(u_0)$ is a representative of $v^{-1}$ and $F(\dot{v}) \in N_{\bG}(\bL)\cap C_{\bG}^{\circ}(u_0)$ is a representative of $F(v)$.
\end{rem}

\begin{pa}\label{pa:Bonnafe-iso-A}
In \cite[\S6.A]{bonnafe:2004:actions-of-rel-Weyl-grps-I} Bonnaf\'{e} shows that for every $v \in W_{\bG}(\bL)$ there exists an isomorphism $\theta_v' : \mathscr{E}_0 \to (\ad\dot{v})^*\mathscr{E}_0$ which induces the identity at the stalk of $u_0$. Clearly this also defines an isomorphism $\mathscr{E}_0\boxtimes\Ql \to (\ad\dot{v})^*(\mathscr{E}_0\boxtimes\Ql)$ which induces the identity at the stalk of any element $u_0z$ where $z\in Z^{\circ}(\bL)$; we will also denote this by $\theta_v'$. As is described in \cref{pa:Lusztig-iso-A} this isomorphism determines a unique endomorphism $\Theta_v'$ of $K$ and by \cite[Proposition 6.1]{bonnafe:2004:actions-of-rel-Weyl-grps-I} we have the map $v\mapsto \Theta_v'$ defines an algebra isomorphism $\Ql W_{\bG}(\bL) \to \mathcal{A}$. The following result describes the relationship between the two sets of basis elements for $\mathcal{A}$.
\end{pa}

\begin{prop}[{}{Bonnaf\'{e}, \cite[Corollary 6.2]{bonnafe:2004:actions-of-rel-Weyl-grps-I}}]\label{prop:bases-of-A}
There exists a linear character $\gamma_{\bL,u_0}^{\bG} \in \Irr(W_{\bG}(\bL))$ such that $\Theta_v' = \gamma_{\bL,u_0}^{\bG}(v)\Theta_v$ and so $\theta_v' = \gamma_{\bL,u_0}^{\bG}(v)\theta_v$ for all $v \in W_{\bG}(\bL)$.
\end{prop}

\begin{pa}\label{pa:connected-centre}
As is remarked in \cite[Remark 6.4]{bonnafe:2004:actions-of-rel-Weyl-grps-I} the linear character $\gamma_{\bL,u_0}^{\bG}$ is known explicitly in almost all cases. Assume that $Z(\bG)$ is connected and $p$ is a good prime for $\bG$. By the reduction arguments given in \cite[\S4.A, \S4.B, Corollary 7.3]{bonnafe:2004:actions-of-rel-Weyl-grps-I} it is sufficient to describe this character when $\bG$ is an adjoint simple group. If $\bG$ is of type $\A_n$ or of exceptional type then $\bL$ is either $\bG$ or a torus and so $\gamma_{\bL,u_0}^{\bG}$ is the trivial character (see \cite[Corollary 6.9]{bonnafe:2004:actions-of-rel-Weyl-grps-I}). When $\bG$ is of type $\B_n$, $\C_n$ or $\D_n$ then in general $\gamma_{\bL,u_0}^{\bG}$ is not the identity. However, this character is described by work of Waldspurger \cite{waldspurger:2001:integrales-orbitales-nilpotentes} as is explained in \cite[Example 6.3]{bonnafe:2004:actions-of-rel-Weyl-grps-I} (see also \cite[Errata]{bonnafe:2005:actions-of-rel-Weyl-grps-II}).
\end{pa}

\section{Rational Structures}\label{sec:rational-structures}
\begin{pa}\label{pa:characteristic-function}
We say a character sheaf $A \in \widehat{\bG}$ is $F$-stable if we have an isomorphism $\phi_A : F^*A \to A$ in $\mathscr{D}\bG$ and we denote by $\widehat{\bG}^F \subseteq \widehat{\bG}$ the set of all $F$-stable character sheaves of $\bG$. Clearly this isomorphism induces a map $\mathscr{H}^iF^*A \to \mathscr{H}^iA$ for each $i \in \mathbb{Z}$ and hence a map $\mathscr{H}^i_xF^*A \to \mathscr{H}^i_xA$ for each $x \in \bG$; we denote both these maps again by $\phi_A$. If $x \in G = \bG^F$ then $\phi_A$ induces an automorphism of the stalk $\mathscr{H}^i_xA$ because $\mathscr{H}^i_xF^*A = \mathscr{H}^i_{F(x)}A = \mathscr{H}^i_xA$. For each such $A$ and $\phi_A$ we then define the \emph{characteristic function} $\chi_{A,\phi_A} : G \to \overline{\mathbb{Q}}_{\ell}$ by setting
\begin{equation*}
\chi_{A,\phi_A}(g) = \sum_i (-1)^i\Tr(\phi_A, \mathscr{H}^i_gA)
\end{equation*}
for all $g \in \bG^F$. Note that this characteristic function depends upon the choice of isomorphism $\phi_A$. If $A$ is an element of $\widehat{\bG}^F$ then we will choose $\phi_A$ to satisfy the properties in \cite[\S25.1]{lusztig:1986:character-sheaves-V}. With such a choice the resulting characteristic function $\chi_{A,\phi_A}$ has norm 1; the choice of $\phi_A$ is unique up to scalar multiplication by a root of unity.
\end{pa}

\begin{pa}\label{pa:F-stable-pairs}
The Frobenius endomorphism acts on the set $\mathcal{N}_{\bG}$ by $\iota \mapsto F^{-1}(\iota) := (F^{-1}(\mathcal{O}_{\iota}),F^*\mathscr{E}_{\iota})$, where $F^*\mathscr{E}_{\iota}$ is the inverse image of $\mathscr{E}_{\iota}$ under $F$. We say $\iota$ is $F$-stable if $F^{-1}(\mathcal{O}_{\iota}) = \mathcal{O}_{\iota}$ and $F^*\mathscr{E}_{\iota}$ is isomorphic to $\mathscr{E}_{\iota}$ (we also denote this by $\iota = F^{-1}(\iota)$); we denote the subset of all $F$-stable pairs by $\mathcal{N}_{\bG}^F$. The Frobenius also acts on the set $\widetilde{\mathcal{M}}_{\bG}$, hence also on $\mathcal{M}_{\bG}$, by $(\bL,\nu) \mapsto (F^{-1}(\bL), F^{-1}(\nu))$. We say $(\bL,\nu)$ is $F$-stable if $F^{-1}(\bL) = \bL$ and $F^{-1}(\nu) = \nu$ and we denote by $\widetilde{\mathcal{M}}_{\bG}^F$ the subset of all $F$-stable pairs (similarly for $\mathcal{M}_{\bG}^F$). Note that the map $\mathcal{N}_{\bG} \to \mathcal{M}_{\bG}$ is compatible with the actions of $F$ so that we have an induced map $\mathcal{N}_{\bG}^F \to \mathcal{M}_{\bG}^F$ (c.f.\ \cite[24.2]{lusztig:1986:character-sheaves-V}). Assume $(\bL,\bP) \in \mathcal{Z}_{\std}$ is such that $\mathcal{N}_{\bL}^0 \neq \emptyset$ then by the classification of cuspidal local systems (see \cite[\S10-15]{lusztig:1984:intersection-cohomology-complexes}) we have $F(\bL) = \bL$ and $F(\bP) = \bP$. In particular any orbit $[\bL,\nu] \in \mathcal{M}_{\bG}^F$ may be represented by an $F$-stable standard Levi subgroup.
\end{pa}

\subsection{Twisted Levi Subgroups}
\begin{pa}\label{lem:G-split-levi}
We now wish to choose for each $F$-stable character sheaf $A \in \widehat{\bG}^F$ satisfying $\supp(A)\cap\bG_{\uni} \neq \emptyset$ a distinguished isomorphism $\phi_A : F^*A \to A$. To do this we will adapt an idea of Lusztig from \cite[\S3.3 - \S3.4]{lusztig:1986:on-the-character-values} and \cite[\S10.3]{lusztig:1985:character-sheaves}. Firstly let us assume that $(\bL,\bP) \in \mathcal{Z}_{\std}$ is a standard pair such that $\mathcal{N}_{\bL}^0 \neq \emptyset$ and $A_0 = A_{\nu}^{\mathscr{L}} \in \widehat{\bL}$ is a cuspidal character sheaf satisfying $(A:\ind_{\bL}^{\bG}A_0) \neq 0$ (see \cref{pa:induction-char-sheaves,lem:F-action-ind}). Using the argument in \cite[\S10.5]{lusztig:1985:character-sheaves} together with \cref{lem:F-action-ind} we see that there exists $x \in \bG$ such that ${}^x\bL$ is $F$-stable and $B_0 = (\ad x^{-1})^*A_0 \in \widehat{{}^x\bL}$ is an $F$-stable cuspidal character sheaf satisfying $(A:\ind_{{}^x\bL}^{\bG}B_0) \neq 0$ (we may then assume that $A$ is a summand of $\ind_{{}^x\bL}^{\bG}B_0$).
\end{pa}

\begin{pa}\label{pa:conj-levi-etc}
Let us assume that $(\bL,\bP) \in \mathcal{Z}_{\std}$ is a standard pair such that $\nu = (\mathcal{O}_0,\mathscr{E}_0) \in \mathcal{N}_{\bL}^0$ and $\mathscr{L} \in \mathcal{S}(\bT_0)$ is a local system. For every $v \in W_{\bG}(\bL)$ we fix an element $g_v \in \bG$ such that $g_v^{-1}F(g_v) = F(\dot{v}^{-1})$ where $\dot{v} \in N_{\bG}(\bL)$ is as in \cref{pa:Lusztig-iso-A}. If $v$ is the identity in $W_{\bG}(\bL)$ then we will assume for convenience that $\dot{v} = g_v$ is the identity in $\bG$ (this clearly satisfies the assumption of \cref{pa:Lusztig-iso-A}). We now define
\begin{equation*}
\bL_v = {}^{g_v}\bL \qquad \bP_v = {}^{g_v}\bP \qquad \mathcal{O}_v = {}^{g_v}\mathcal{O}_0 \qquad \Sigma_v = \mathcal{O}_vZ^{\circ}(\bL_v) \qquad \mathscr{E}_v = (\ad g_v^{-1})^*\mathscr{E}_0 \qquad \mathscr{L}_v = (\ad g_v^{-1})^*\mathscr{L}.
\end{equation*}
Through the isomorphism $\ad g_v^{-1} : \bL_v \to \bL$ we may identify the action of the Frobenius endomorphism $F$ on $\bL_v$ with the action of $F_v$ on $\bL$ defined by $F_v(l) = F(\dot{v}^{-1}l\dot{v})$ for all $l \in \bL$. Clearly we have an isomorphism of abstract groups $\ad g_v^{-1} : \bL_v^F \to \bL^{F_v}$ and $(\ad g_v^{-1})^*$ induces a bijective correspondence between $\widehat{\bL}^{F_v}$ and $\widehat{\bL}_v^F$. In particular, assume $B \in \widehat{\bL}^{F_v}$ and $\psi : F_v^*B \to B$ is a fixed isomorphism then $\psi' : F^*B' \to B'$ is also an isomorphism where $\psi' = (\ad g_v^{-1})^*\psi$ and $B' = (\ad g_v^{-1})^*B$. From the definitions we have the corresponding characteristic functions are related by $\chi_{B',\psi'}= \chi_{B,\psi} \circ \ad g_v^{-1}$.


Let us denote by $D_v$ the complex $\IC(\overline{\Sigma}_v,\mathscr{E}_v\boxtimes\mathscr{L}_v)[\dim \Sigma_v]$ (see \cref{eq:cusp-uni-sup}). From the discussion in \cref{lem:G-split-levi} we need only concern ourselves with summands of $\ind_{\bL_v}^{\bG}(D_v)$ where $D_v$ is $F$-stable. Clearly we have
\begin{equation*}
F^*D_v = \IC(\overline{F^{-1}(\Sigma_v)},F^*\mathscr{E}_v\boxtimes F^*\mathscr{L}_v)[\dim \Sigma_v]
\end{equation*}
hence we may, and will, assume that $\nu \in \mathcal{N}_{\bL}^F$ and $F^*\mathscr{L}_v \cong \mathscr{L}_v$ (see \cite[Theorem 9.2(b)]{lusztig:1984:intersection-cohomology-complexes}). As we have an equivalence $F^*\mathscr{L}_v \cong \mathscr{L}_v \Leftrightarrow F_v^*\mathscr{L} \cong \mathscr{L}$ we see that the existence of an isomorphism $F^*\mathscr{L}_v \cong \mathscr{L}_v$ is equivalent to $v^{-1}$ being contained in the subset
\begin{equation}\label{eq:F-coset-loc-sys}
Z_{\bG}(\bL,\mathscr{L}) = \{n \in N_{\bG}(\bL) \mid \ad(n)^*F^*\mathscr{L} \cong \mathscr{L}\}/\bL \subset W_{\bG}(\bL).
\end{equation}
In particular we will assume that $Z_{\bG}(\bL,\mathscr{L})$ is non-empty. Note that $Z_{\bG}(\bL,\mathscr{L})x = Z_{\bG}(\bL,\mathscr{L})$ for any $x \in W_{\bG}(\bL,\mathscr{L})$, hence $Z_{\bG}(\bL,\mathscr{L})$ is a union of right cosets of $W_{\bG}(\bL,\mathscr{L})$ in $W_{\bG}(\bL)$.

If $\bS \leqslant \bG$ is any $F$-stable torus and $\mathscr{F} \in \mathcal{S}(\bS)^F$ is any $F$-stable local system then we may choose a distinguished isomorphism $\beta : F^*\mathscr{F} \to \mathscr{F}$ by the condition that the morphism induced by $\beta$ at the stalk of the identity element $1 \in \bS^F$ is the identity morphism. In particular this defines an isomorphism $\varphi_1^v : F^*\mathscr{L}_v \to \mathscr{L}_v$ for each $v^{-1} \in Z_{\bG}(\bL,\mathscr{L})$.
\end{pa}


\subsection{Relative Weyl Groups}
\begin{pa}\label{pa:generators-rel-Weyl}
For the following we refer to \cite[Proposition 1.12]{bonnafe:2004:actions-of-rel-Weyl-grps-I} and the references therein (in particular \cite{howlett:1980:normalizers-of-parabolic-subgroups}). Let us denote by $\Phi_{\bL} \subseteq \Phi_{\bG}$ the root system of $\bL$ with respect to $\bT_0$. We will denote by $V_{\bG}$ the $\mathbb{R}$-subspace of $\mathbb{R} \otimes_{\mathbb{Z}} X(\bT_0)$ spanned by $\Phi_{\bG}$ and by $V_{\bL} \subseteq V_{\bG}$ the $\mathbb{R}$-subspace spanned by $\Phi_{\bL}$. Let $\Delta_{\bL} \subset \Delta_{\bG}$ be the set of simple roots of $\bL$ with respect to $\bT_0\leqslant \bB_0 \cap\bL$ then we denote by $\overline{\Delta}_{\bL}$ the set $\Delta_{\bG} - \Delta_{\bL}$. For each $\alpha \in \overline{\Delta}_{\bL}$ let $\bM_{\alpha} \in \mathcal{L}_{\std}$ be the standard Levi subgroup of $\bG$ whose simple roots are $\Delta_{\bL}\cup\{\alpha\}$ with respect to $\bT_0\leqslant\bB_0\cap\bM_{\alpha}$ then the group $W_{\bM_{\alpha}}(\bL)$ contains a unique non-trivial element which we denote by $s_{\bL,\alpha}$. For every $\alpha \in \overline{\Delta}_{\bL}$ we have $s_{\bL,\alpha}$ acts on the quotient space $\overline{V}_{\bL} = V_{\bG}/V_{\bL}$ as a reflection sending $\alpha + V_{\bL}$ to $-\alpha + V_{\bL}$. We will also use $\overline{\Delta}_{\bL}$ to denote the image $\{\alpha + V_{\bL} \mid \alpha \in \widetilde{\Delta}_{\bL}\}$ of $\overline{\Delta}_{\bL}$ in $\overline{V}_{\bL}$. Taking $\mathbb{I} = \{s_{\bL,\alpha} \mid \alpha \in \overline{\Delta}_{\bL}\}$ we have $(W_{\bG}(\bL),\mathbb{I})$ is a Coxeter system with corresponding root system $\overline{\Phi}_{\bL} \subset \overline{V}_{\bL}$ obtained as the image of $\overline{\Delta}_{\bL}$ under the action of $W_{\bG}(\bL)$. Note that if $\overline{\Phi}_{\bL}^+\subset\overline{\Phi}_{\bL}$ is the positive system of roots determined by $\overline{\Delta}_{\bL}$ then we have $\overline{\Phi}_{\bL}^+ = \overline{\Phi}_{\bL} \cap \{\alpha + V_{\bL} \mid \alpha \in \Phi^+\}$.

For any root $\alpha \in \overline{\Phi}_{\bL}$ we denote by $n_{\alpha} \in N_{\bG}(\bL)$ a representative for the reflection $s_{\alpha} \in W_{\bG}(\bL)$ of $\alpha$. We now consider the following subgroup
\begin{equation*}
R_{\bG}(\bL,\mathscr{L}) = \langle s_{\alpha} \in W_{\bG}(\bL) \mid \alpha \in \overline{\Phi}_{\bL}\text{ and }\ad(n_{\alpha})^*\mathscr{L} \cong \mathscr{L}\rangle
\end{equation*}
of $W_{\bG}(\bL,\mathscr{L})$. Clearly this is a reflection subgroup of $W_{\bG}(\bL)$ and the defining condition does not depend upon the choice of representative $n_{\alpha}$. In particular there exists a root system $\Psi \subset \overline{\Phi}_{\bL}$ such that $R_{\bG}(\bL,\mathscr{L})$ is the reflection group of $\Psi$. Taking $\Psi^+ = \Psi\cap\overline{\Phi}_{\bL}^+$ we have $\Psi^+$ is a system of positive roots in $\Psi$ which determines a unique set of Coxeter generators $\mathbb{J} \subseteq R_{\bG}(\bL,\mathscr{L})$.
\end{pa}

%

\begin{pa}\label{pa:min-length-elm}
Assume $n \in N_{\bG}(\bL,\mathscr{L})$ then for any $v^{-1} \in Z_{\bG}(\bL,\mathscr{L})$ we have an isomorphism
\begin{equation*}
(\ad F_v^{-1}(n))^*\mathscr{L}\cong (\ad F_v^{-1}(n))^*(F_v)^*\mathscr{L} = F_v^*(\ad n)^*\mathscr{L} \cong F_v^*\mathscr{L} \cong \mathscr{L},
\end{equation*}
in particular as $\bL$ is $F_v$-stable this shows that $F_v$ induces an automorphism $F_v : W_{\bG}(\bL,\mathscr{L}) \to W_{\bG}(\bL,\mathscr{L})$. Assume $w^{-1} \in Z_{\bG}(\bL,\mathscr{L})$ is the unique element of minimal length in its right coset $w^{-1}R_{\bG}(\bL,\mathscr{L})$ with respect to the length function of $(W_{\bG}(\bL),\mathbb{I})$ (see \cite[Lemma 1.9(i)]{lusztig:1984:characters-of-reductive-groups}). This element is uniquely determined in its coset by the condition that $w^{-1}\cdot\Psi^+ \subset \Phi^+$. In particular, for such an element, $F_w$ is an automorphism of $W_{\bG}(\bL,\mathscr{L})$ which restricts to an automorphism of the Coxeter system $(R_{\bG}(\bL,\mathscr{L}),\mathbb{J})$. Taking $R_{\bG}(\bL_w,\mathscr{L}_w)$ (resp.\ $\mathbb{J}_w$) to be the image of $R_{\bG}(\bL,\mathscr{L})$ (resp.\ $\mathbb{J}$) under $\ad g_w$ we have $(R_{\bG}(\bL_w,\mathscr{L}_w),\mathbb{J}_w)$ is a Coxeter system and $F$ induces an automorphism of this Coxeter system. Note that this automorphism depends upon the choice of the element $w \in Z_{\bG}(\bL,\mathscr{L})$.
\end{pa}

\begin{rem}
Assume $Z(\bG)$ is connected and $\bG/Z(\bG)$ is simple then by \cite[I - 5.16.1, II - 4.2]{shoji:1995:character-sheaves-and-almost-characters} and  \cite[8.5.13]{lusztig:1984:characters-of-reductive-groups} we have $W_{\bG}(\bL,\mathscr{L}) = R_{\bG}(\bL,\mathscr{L})$ is a reflection subgroup of $W_{\bG}(\bL)$. Furthermore $Z_{\bG}(\bL,\mathscr{L})$ is a single right coset of $W_{\bG}(\bL,\mathscr{L})$ in $W_{\bG}(\bL)$, hence the element $w$ chosen above is unique. It is well known that if $Z(\bG)$ is disconnected then these statements do not hold (for instance when $\bL$ is a torus).
\end{rem}


\subsection{Isomorphisms for Local Systems}
\begin{pa}\label{pa:isomorphism-cuspidal-case}
We now consider how to choose an isomorphism $F^*\mathscr{E}_v \to \mathscr{E}_v$ for any $v \in W_{\bG}(\bL)$. First we will choose an isomorphism $\varphi_0 : F^*\mathscr{E}_0 \to \mathscr{E}_0$ such that the induced isomorphism $(\mathscr{E}_0)_u \to (\mathscr{E}_0)_u$ at the stalk of any element $u \in \mathcal{O}_0^F$ is $q^{(\dim(\bL/Z^{\circ}(\bL)) - \dim\mathcal{O}_0)/2}$ times a map of finite order. Following \cite[9.3]{lusztig:1990:green-functions-and-character-sheaves} we define an isomorphism $\varphi_0^v : F^*\mathscr{E}_v \to \mathscr{E}_v$ in the following way. Let $K$ be the image of $(\bL,\mathcal{O}_0,\mathscr{E}_0,\mathscr{L})$ under the map in \cref{eq:ind-w-P-construction} and let $\mathcal{A}$ be its endomorphism algebra (c.f.\ \cref{pa:endomorphism-algebra-A}). The basis element $\theta_{F(v)}$ of $\mathcal{A}$ defines an isomorphism $\theta_{F(v)} : \mathscr{E}_0 \to (\ad F(\dot{v}))^*\mathscr{E}_0$. Taking the inverse image of $\theta_{F(v)}$ along $\ad g_v^{-1}\circ F$ we see that we obtain a new isomorphism
\begin{equation*}
(\ad g_v^{-1}\circ F)^*\theta_{F(v)} : F^*\mathscr{E}_v \to (\ad g_v^{-1})^*F^*\mathscr{E}_0.
\end{equation*}
Moreover the isomorphism $\varphi_0$ induces an isomorphism
\begin{equation*}
(\ad g_v^{-1})^*\varphi_0 : (\ad g_v^{-1})^*F^*\mathscr{E}_0 \to \mathscr{E}_v.
\end{equation*}
We will now take $\varphi_0^v$ to be the composition $(\ad g_v^{-1})^*\varphi_0\circ(\ad g_v^{-1}\circ F)^*\theta_{F(v)}$ (note that by our assumption on $g_v$ we have $\varphi_0^v = \varphi_0$ if $v=1$). If $v^{-1} \in Z_{\bG}(\bL,\mathscr{L})$ then we can define a unique isomorphism $\phi_{D_v} : F^*D_v \to D_v$ with the property that the restriction of $\phi_{D_v}$ to $\Sigma_v$ coincides with $\varphi_0^v\boxtimes\varphi_1^v$ (c.f.\ \cref{pa:conj-levi-etc}).
\end{pa}


\begin{rem}
Note that the integer $\dim(\bL/Z^{\circ}(\bL)) - \dim\mathcal{O}_0$ is not even in general. Hence our construction above depends upon a choice of $q^{1/2}$ in $\Ql$. However, if $Z(\bG)$ is connected and $p$ is a good prime then this is an integer by the classification of cuspidal pairs (see \cite{lusztig:1984:intersection-cohomology-complexes}).
\end{rem}

\begin{rem}\label{rem:scaled}
Here, we have not quite followed the setup of \cite{lusztig:1986:character-sheaves-V} because we have scaled the map of finite order, chosen in \cite[24.2]{lusztig:1986:character-sheaves-V}, by a power of $q$. When following \cite{lusztig:1986:character-sheaves-V} we must then adjust all terms by this scalar, which the reader may check that we do.
\end{rem}

\subsection{Isomorphisms for Induced Complexes}
\begin{pa}\label{pa:iso-K-v}
Assume $v^{-1} \in Z_{\bG}(\bL,\mathscr{L})$ and let us denote by $K_v$ the image of $(\bL_v,\mathcal{O}_v,\mathscr{E}_v,\mathscr{L}_v)$ under the map in \cref{eq:ind-w-P-construction}. We now consider how to choose an isomorphism $F^*K_v \to K_v$. Recalling the construction in \cref{pa:unip-sup-char-sheaves} we see that the isomorphisms $\varphi_0^v : F^*\mathscr{E}_v \to \mathscr{E}_v$ and $\varphi_1^v : F^*\mathscr{L}_v \to \mathscr{L}_v$ naturally induce isomorphisms $\tilde{\varphi}_0^v : F^*\tilde{\mathscr{E}}_v \to \tilde{\mathscr{E}}_v$ and $\tilde{\varphi}_1^v : F^*\tilde{\mathscr{L}}_v \to \tilde{\mathscr{L}}_v$ hence an isomorphism $\tilde{\varphi}_0^v\otimes\tilde{\varphi}_1^v : F^*\tilde{\mathscr{E}}_v\otimes F^*\tilde{\mathscr{L}}_v \to \tilde{\mathscr{E}}_v\otimes\tilde{\mathscr{L}}_v$. Clearly the variety $Y$ of \cref{pa:unip-sup-char-sheaves} is stable under $F$ (because $\bL_v$ is stable under $F$) so we have
\begin{equation*}
F^*K_v = \IC(\overline{Y},F^*\gamma_*(\tilde{\mathscr{E}}_v\otimes\tilde{\mathscr{L}}_v))[\dim Y] = \IC(\overline{Y},\gamma_*F^*(\tilde{\mathscr{E}}_v\otimes\tilde{\mathscr{L}}_v))[\dim Y]
\end{equation*}
because $\gamma_* = \gamma_!$ and $F^*\gamma_! = \gamma_!F^*$ (see \cref{pa:unip-sup-char-sheaves} and \cref{pa:direct-image-commutative}). We may now define a unique isomorphism $\phi : F^*K_v \to K_v$ by specifying that $\phi|_Y$ coincides with $\gamma_*(\tilde{\varphi}_0^v\otimes\tilde{\varphi}_1^v)$.
\end{pa}

\subsection{Isomorphisms for Character Sheaves}
\begin{pa}\label{pa:iso-on-arbitrary-char-sheaf}
Let $K_v$ and $\phi : F^*K_v \to K_v$ be as in \cref{pa:iso-K-v}. By \cref{prop:K-nu-L-iso-to-induction} we may assume that any summand $A$ of $\ind_{\bL_v}^{\bG}(D_v)$ is a summand of $K_v \cong \ind_{\bL_v}^{\bG}(D_v)$. In particular, if $\mathcal{A}_v$ is the endomorphism algebra of $K_v$ then $A = K_{v,E} = (K_v)_E$ for some simple $\mathcal{A}_v$-module $E$ (c.f.\ \cref{pa:decomp-by-end-algebra}). By adapting the construction in \cite[\S10.3]{lusztig:1985:character-sheaves} we will show how $\phi_A$ is determined from $\phi$. Let us denote by $\sigma : \mathcal{A}_v \to \mathcal{A}_v$ the algebra automorphism given by $\sigma(\theta) = \phi \circ F^*\theta \circ \phi^{-1}$ where $F^*\theta : F^*K_v \to F^*K_v$ is the induced map. For any $\mathcal{A}_v$-module $E$ we denote by $E_{\sigma}$ the module obtained from $E$ by twisting the action with $\sigma^{-1}$ (c.f.\ \cref{pa:autos-of-fin-groups}). For each such $E$ we then have an induced isomorphism of $\mathcal{A}_v$-modules
\begin{equation*}
F^*A = \Hom_{\mathcal{A}}(E,F^*K_v) \to \Hom_{\mathcal{A}}(E_{\sigma},K_v)
\end{equation*}
given by $f \mapsto \phi\circ f$. Here $F^*K_v$ is an $\mathcal{A}$-module under the action $\theta\cdot k = (F^*\theta)(k)$ for all $k \in F^*K_v$ and the first equality is seen to hold by the construction given in \cref{eq:diag-Hom-MG}.

Assume we have an $\mathcal{A}_v$-module isomorphism $\psi_E : E \to E_{\sigma}$ then $\phi_A : F^*K_{v,E} \to K_{v,E}$, given by $\phi_A(f) = \phi\circ f \circ \psi_E$, is an isomorphism of $\mathcal{A}_v$-modules which defines an isomorphism in $\mathscr{D}\bG$ under the construction in \cref{eq:diag-Hom-MG}. We now need only observe that when $K_{v,E}$ is simple, hence $E$ is simple, all such isomorphisms occur in this way. To see this note that any non-zero $f \in \Hom_{\mathcal{A}_v}(E,K_v)$ gives an isomorphism $E \cong \Image(f)$, because $E$ is a simple module, so $\psi_E = f^{-1}\circ\phi^{-1}\circ\phi_A\circ f$ is determined by $\phi_A$. This discussion shows that choosing the isomorphism $\phi_A$ is equivalent to choosing the isomorphism $\psi_E$.

Let us recall Lusztig's basis $\{\Theta_y \mid y \in W_{\bG}(\bL_v,\mathscr{L}_v)\}$ for the endomorphism algebra $\mathcal{A}_v$ (c.f.\ \cref{pa:Lusztig-iso-A}). From their definition one may readily check that the basis elements of $\mathcal{A}_v$ satisfy $\sigma(\Theta_y) = \Theta_{F^{-1}(y)}$ for all $y \in W_{\bG}(\bL_v,\mathscr{L}_v)$. Let $\widetilde{W}_{\bG}(\bL_v,\mathscr{L}_v)$ denote the semidirect product $W_{\bG}(\bL_v,\mathscr{L}_v)\rtimes \langle F\rangle$ where $\langle F\rangle$ is the finite cyclic group generated by the automorphism $F$. We can extend $E$ to a $\widetilde{W}_{\bG}(\bL_v,\mathscr{L}_v)$-module $\widetilde{E}$ by letting $F^{-1}$ act as $\psi_E$ under the isomorphism in \cref{eq:end-grp-alg-iso}. In this way we see that choosing $\phi_A$ is equivalent to choosing an extension $\widetilde{E} \in \Irr(\widetilde{W}_{\bG}(\bL_v,\mathscr{L}_v))$ of $E$.
\end{pa}

\begin{rem}\label{rem:preferred-extension}
Assume $\bW$ is a finite Weyl group, $(\bW, \mathbb{T})$ is a finite Coxeter system and $\phi : \bW \to \bW$ is an automorphism stabilising $\mathbb{T}$. For any $\phi$-stable irreducible character $\chi \in \Irr(\bW)^{\phi}$ Lusztig has systematically defined a non-canonical extension of $\chi$ to the semidirect product $\widetilde{\bW} \rtimes \langle \phi\rangle$ (see \cite[17.2]{lusztig:1985:character-sheaves}). Throughout this article we will refer to this extension as \emph{Lusztig's preferred extension} or simply the \emph{preferred extension}.
\end{rem}

\begin{assumption}
Assume now that $v = w$, where $w^{-1} \in Z_{\bG}(\bL,\mathscr{L})$ is as in \cref{pa:min-length-elm}, and $R_{\bG}(\bL_w,\mathscr{L}_w) = W_{\bG}(\bL_w,\mathscr{L}_w)$ is a Coxeter group then $F$ is an automorphism of the Coxeter system $(W_{\bG}(\bL_w,\mathscr{L}_w), \mathbb{J}_w)$ (c.f.\ \cref{pa:generators-rel-Weyl}). In this situation we will assume that $\psi_E$ is chosen such that $\widetilde{E}$ is Lusztig's preferred extension of $E$.
\end{assumption}

\subsection{Passing from Twisted to Split Levis on the Trivial Local System}
\begin{pa}\label{pa:iso-K-v-compare}
Assume, only for this section, that $\mathscr{L} = \Ql$. Let $K_v$ be the image of $(\bL_v,\mathcal{O}_v,\mathscr{E}_v,\mathscr{L}_v)$ under the map in \cref{eq:ind-w-P-construction} and let $\psi^v : F^*K_v \to K_v$ be the isomorphism defined in \cref{pa:iso-K-v}. We will denote by $K$ the image of $(\bL,\mathcal{O}_0,\mathscr{E}_0,\mathscr{L})$ under the map in \cref{eq:ind-w-P-construction}. We wish to define an isomorphism $\phi^v : F^*K \to K$ which is related (through $\Theta_{F(v)}$) to the isomorphism $\psi^v$. Recall that in \cref{pa:isomorphism-cuspidal-case,pa:iso-K-v} we have defined isomorphisms $\phi : F^*K\to K$, $\varphi_0 : F^*\mathscr{E}_0 \to \mathscr{E}_0$, $\varphi_1 : F^*\mathscr{L} \to \mathscr{L}$ and $\tilde{\varphi}_0 : F^*\tilde{\mathscr{E}}_0 \to \tilde{\mathscr{E}}_0$ (simply take $v=1$ in the constructions). With this we obtain isomorphisms
\begin{gather*}
\varphi_{v,0} := \theta_{F(v)}(\varphi_0\boxtimes\varphi_1) : F^*(\mathscr{E}_0\boxtimes\Ql) \to (\ad \dot{v})^*(\mathscr{E}_0\boxtimes
\Ql)\\
\tilde{\varphi}_{v,0} := \tilde{\theta}_{F(v)}\tilde{\varphi}_0 : F^*\tilde{\mathscr{E}}_0 \to \gamma_v^*\tilde{\mathscr{E}}_0
\end{gather*}
by composing with the endomorphism defined in \cref{pa:Lusztig-iso-A}. As in \cref{pa:iso-K-v} we define the isomorphism $\phi^v$ to be the unique extension of the isomorphism $\gamma_*\tilde{\varphi}_{v,0}$. It is clear from the construction that we have $\phi^v = \Theta_{F(v)}\phi$ and furthermore using the exact same argument as in \cite[(10.6.1)]{lusztig:1985:character-sheaves} we have for each $g \in G$ and $i \in \mathbb{Z}$ that
\begin{equation}\label{eq:trace-comparison}
\Tr(\phi^v,\mathscr{H}_g^iK) = \Tr(\Theta_{F(v)}\phi,\mathscr{H}_g^iK) = \Tr(\psi^v,\mathscr{H}_g^iK_v).
\end{equation}

Assume now that $K_E$ is the $F$-stable summand of $K$ parameterised by the irreducible character $E \in \Irr(W_{\bG}(\bL))^F$ (c.f.\ \cref{pa:endomorphism-algebra-A}). By \cref{pa:iso-on-arbitrary-char-sheaf} we have an isomorphism $\phi_E : F^*K_E \to K_E$ defined by $\phi$ and the choice of extension $\widetilde{E}$. Replacing $\phi$ by $\phi^v$ in the construction of \cref{pa:iso-on-arbitrary-char-sheaf} we see that the isomorphism $\phi_E$ is replaced by $\phi_E^v = \Theta_{F(v)}\phi_E$. Note that we have a bijection
\begin{equation}\label{eq:bijection-irr-W_G(L)}
\Irr(W_{\bG}(\bL))^F = \Irr(W_{\bG}(\bL))^{F_v} \to \Irr(W_{\bG}(\bL_v))^F,
\end{equation}
induced by the isomorphism $\ad g_v$. We will denote by $K_{v,E}$ the summand of $K_v$ which is parameterised by $E' \in \Irr(W_{\bG}(\bL_v))^F$ corresponding to $E$ under \cref{eq:bijection-irr-W_G(L)}. By \cref{pa:iso-on-arbitrary-char-sheaf} we have an isomorphism $\psi_E^v : F^*K_{v,E} \to K_{v,E}$ whose definition depends upon $\psi^v$ and an extension $\widetilde{E}'$ of $E'$. We may, and will, assume that the restriction of $\widetilde{E}'$ to the coset $W_{\bG}(\bL_v).F$ coincides with the restriction of $\widetilde{E}$ to the coset $W_{\bG}(\bL).F$ under the correspondence in \cref{pa:conventions-coset-identification}. Arranging things in this way, and using \cref{eq:trace-comparison}, we then have for each $g \in G$ and $i \in \mathbb{Z}$ that
\begin{equation}\label{eq:trace-comparison-iots}
\Tr(\phi_E^v,\mathscr{H}_g^iK_E) = \Tr(\Theta_{F(v)}\phi_E,\mathscr{H}_g^iK_{\iota}) = \Tr(\psi_E^v,\mathscr{H}_g^iK_{v,E}).
\end{equation}
\end{pa}

\subsection{Bases of the Space of Unipotently Supported Class Functions}
\begin{pa}\label{pa:a-and-b-values}
Before we continue we define here two integer values associated to a pair $\iota \in \mathcal{N}_{\bG}$. Recall that $[\bL_{\iota},\nu_{\iota}] \in \mathcal{M}_{\bG}$ is the orbit such that $\iota \in \mathscr{I}[\bL_{\iota},\nu_{\iota}]$ (c.f.\ \cref{pa:gen-spring-cor}). Let us denote $\nu_{\iota}$ by $(\mathcal{O}_0,\mathscr{E}_0)$ then we attach to $\iota$ the integers
\begin{align*}
a_{\iota} &= -\dim\mathcal{O}_{\iota} - \dim Z^{\circ}(\bL_{\iota}),\\
b_{\iota} &= (\dim\bG - \dim\mathcal{O}_{\iota}) - (\dim\bL_{\iota} - \dim\mathcal{O}_0).
\end{align*}
We now consider how the above isomorphisms determine a canonical isomorphism $F^*\mathscr{E}_{\iota} \to \mathscr{E}_{\iota}$ for an arbitrary pair $\iota = (\mathcal{O}_{\iota},\mathscr{E}_{\iota}) \in \mathcal{N}_{\bG}^F$. Recall from \cite[Theorem 6.5(c)]{lusztig:1984:intersection-cohomology-complexes} that the complex $K_{\iota}$ (c.f.\ \cref{pa:endomorphism-algebra-A}) is such that $\mathscr{H}^{a_{\iota}}(K_{\iota})|_{\mathcal{O}_{\iota}} \cong \mathscr{E}_{\iota}$. Assume $\iota$ corresponds to $E \in \Irr(W_{\bG}(\bL_{\iota}))^F$ under the generalised Springer correspondence. In \cref{pa:iso-K-v-compare} we have defined an isomorphism $\phi_E^v : F^*K_E \to K_E$ (for each $v \in W_{\bG}(\bL_{\iota})$) which naturally induces an isomorphism $\phi_{\iota}^v : F^*K_{\iota} \to K_{\iota}$ by the definition of $K_{\iota}$. We similarly obtain an induced isomorphism $F^*\mathscr{H}^{a_{\iota}}(K_{\iota})|_{\mathcal{O}_{\iota}} \to \mathscr{H}^{a_{\iota}}(K_{\iota})|_{\mathcal{O}_{\iota}}$ hence an isomorphism $\varphi_{\iota}^v : F^*\mathscr{E}_{\iota} \to \mathscr{E}_{\iota}$. Setting $v=1$ gives the required canonical isomorphism. Following \cite[24.2.2]{lusztig:1986:character-sheaves-V} we define, for every $v \in W_{\bG}(\bL_{\iota})$, an isomorphism $\psi_{\iota}^v$ by the condition that $\varphi_{\iota}^v = q^{(\dim\bG+a_{\iota})/2}\psi_{\iota}^v$. Note that we are using \cref{rem:scaled} and the fact that $b_{\iota} = a_{\iota} + \dim\supp K_{\iota}$. This latter statement follows from the fact that $\dim\supp K_{\iota} = \dim Y$ where $Y$ is as in \cref{pa:K-nu-L-ind-A-nu-L}.
\end{pa}

\begin{pa}\label{pa:X-and-Y}
Let us assume that to every element $\iota \in \mathcal{N}_{\bG}$ we have associated an element $v_{\iota} \in W_{\bG}(\bL_{\iota})$ such that this assignment is constant on blocks (i.e.\ $v_{\iota} = v_{\iota'}$ if $\iota$ and $\iota'$ are contained in the same block of $\mathcal{N}_{\bG}$). Following \cite[\S24.2]{lusztig:1986:character-sheaves-V} we associate to each $F$-stable pair $\iota \in \mathcal{N}_{\bG}^F$ a unipotently supported class function of $G$ by setting
\begin{align*}\label{eq:def-X_i}
X_{\iota}^v(g) &= (-1)^{a_{\iota}}q^{-(\dim\bG + a_{\iota})/2}\chi_{K_{\iota},\phi_{\iota}^v}(g),
\intertext{for all $g \in G$ (where $\phi_{\iota}^v$ is as in \cref{pa:iso-K-v-compare}). Here we simply write $v$ for $v_{\iota}$ with the meaning understood from the context. Furthermore for each $\iota \in \mathcal{N}_{\bG}^F$ we define a second unipotently supported class function of $G$ by setting}
Y_{\iota}^v(g) &= \begin{cases}
\Tr(\psi_{\iota}^v,(\mathscr{E}_{\iota})_g) &\text{if }g \in \mathcal{O}_{\iota}^F,\\
0 &\text{otherwise},
\end{cases}
\end{align*}
for all $g \in G$ (where $\psi_{\iota}^v$ is as in \cref{pa:a-and-b-values}). The sets $\mathcal{Y}^* = \{Y_{\iota}^v \mid \iota \in \mathcal{N}_{\bG}^F\}$ and $\mathcal{X}^* = \{X_{\iota}^v \mid \iota \in \mathcal{N}_{\bG}^F\}$ are bases for the subspace $\Centu{G}$ of unipotently supported class functions of $G$ (see \cite[24.2.7]{lusztig:1986:character-sheaves-V}). In particular for each $\iota$, $\iota' \in \mathcal{N}_{\bG}^F$ there exists an element $P_{\iota',\iota} \in \overline{\mathbb{Q}}_{\ell}$ such that
\begin{equation*}
X_{\iota}^v = \sum_{\iota' \in \mathcal{N}_{\bG}^F} P_{\iota',\iota}Y_{\iota'}^v.
\end{equation*}
The matrix $(P_{\iota',\iota})_{\iota',\iota \in \mathcal{N}_{\bG}^F}$ is precisely the matrix considered in \cite[\S24]{lusztig:1986:character-sheaves-V} and is computable by the algorithm described in the proof of \cite[Theorem 24.4]{lusztig:1986:character-sheaves-V}. We recall that the following property of the coefficients $P_{\iota',\iota}$ holds
\begin{equation}\label{eq:P-properties}
P_{\iota'\iota} = 0 \text{ if }\mathscr{I}[\bL_{\iota},\nu_{\iota}] \neq \mathscr{I}[\bL_{\iota'},\nu_{\iota'}].
\end{equation}
\end{pa}

\begin{assumption}
From now on the following assumption is in place. If $A \in \widehat{\bG}^F$ is a character sheaf satisfying $\supp(A)\cap\bG_{\uni} \neq \emptyset$ then we assume the isomorphism $\phi_A : F^*A \to A$ to be chosen as in \cref{pa:iso-on-arbitrary-char-sheaf}. If $\iota \in \mathcal{N}_{\bG}^F$ is any $F$-stable pair and $v \in W_{\bG}(\bL_{\iota})$ then we assume the isomorphism $\phi_{\iota}^v : F^*K_{\iota} \to K_{\iota}$ to be as in \cref{pa:a-and-b-values}.
\end{assumption}

\section{Restricting Character Sheaves to the Unipotent Variety}\label{sec:restricting-to-unip-variety}
\begin{pa}
Let us carry forward the setup of the previous section. We will denote by $K_w^{\mathscr{L}}$ the image of $(\bL_w,\mathcal{O}_w,\mathscr{E}_w,\mathscr{L}_w)$ under the map in \cref{eq:ind-w-P-construction}, where $w \in W_{\bG}(\bL)$ is as in \cref{pa:min-length-elm}. Furthermore, we assume $A = (K_w^{\mathscr{L}})_E = K_{w,E}^{\mathscr{L}}$ is the summand of $K_w^{\mathscr{L}}$ parameterised by $E \in \Irr(W_{\bG}(\bL_w))^F$ (c.f.\ \cref{pa:decomp-by-end-algebra}). Setting $\mathscr{L} = \Ql$ we obtain a new complex which we denote simply by $K_w$. If $f \in \Cent(G)$ is a class function then we denote by $f|_{G_{\uni}} \in \Centu{G}$ the extension by 0 of the restriction to $G_{\uni}$. Recall from \cref{pa:X-and-Y} that $\mathcal{X}^*$ is a basis of $\Centu{G}$ then we have
\begin{equation*}
\chi_{A,\phi_A}|_{G_{\uni}} = \sum_{\iota \in \mathcal{N}_{\bG}^F} m(A,\iota,\phi_{\iota}^v)\chi_{K_{\iota},\phi_{\iota}^v}
\end{equation*}
for some coefficients $m(A,\iota,\phi_{\iota}^v) \in \Ql$. We will assume that the element $v_{\iota}$ chosen in \cref{pa:X-and-Y} is $w$ when $\iota$ is contained in the block $\mathscr{I}[\bL,\nu]$ ($\nu = (\mathcal{O}_0,\mathscr{E}_0)$). It is the purpose of this section to describe the coefficients $m(A,\iota,\phi_{\iota}^v)$. We will do this following the method in \cite{lusztig:1986:on-the-character-values}, in particular we will now recall the sequence of isomorphisms constructed in \cite[\S2.6]{lusztig:1986:on-the-character-values}.
\end{pa}

\begin{pa}\label{pa:d_w-iso}
Assume $(\bL_w,\bQ) \in \mathcal{Z}$ then we will denote by $\ind_{\bL_w \subseteq \bQ}^{\bG}(A_{\mathscr{L}})$ the image of $(\bL_w,\bQ,\mathcal{O}_w,\mathscr{E}_w,\mathscr{L}_w)$ under the map in \cref{eq:ind-P-construction}. Let $D$ be the similarly named complex defined in \cref{pa:definition-of-induction} then by the discussion in \cref{pa:K-nu-L-ind-A-nu-L} we have $D = \IC(\tilde{X}',\overline{\mathscr{E}_w\boxtimes\mathscr{L}_w})[\dim \tilde{X}']$. Setting $\mathscr{L} = \Ql$ we obtain new complexes which we respectively denote by $D_0$ and $\ind_{\bL_w \subseteq \bQ}^{\bG}(A_0)$. Recall the notation of \cref{pa:K-nu-L-ind-A-nu-L} and let $i : \overline{Y}_{\uni} \hookrightarrow \overline{Y}$ be the natural inclusion of the unipotent elements contained in $\overline{Y}$. We have a commutative diagram
\begin{center}
\begin{tikzcd}
\overline{\mathcal{O}}_0 \arrow{d}[swap]{i} & \hat{X}_{\uni}' \arrow{l}[swap]{\pi}\arrow{d}[swap]{i}\arrow{r}{\sigma} & \tilde{X}_{\uni}' \arrow{d}[swap]{i}\arrow{r}{\tau} & \overline{Y}_{\uni} \arrow{d}[swap]{i}\\
\overline{\Sigma} & \hat{X}' \arrow{l}[swap]{\pi} \arrow{r}{\sigma} & \tilde{X}' \arrow{r}{\tau} & \overline{Y}
\end{tikzcd}
\end{center}
where $\tilde{X}_{\uni}' = \tau^{-1}(\overline{Y}_{\uni})$ and $\hat{X}_{\uni}' = \sigma^{-1}(\tilde{X}_{\uni}')$. For $\tilde{X}_{\uni}'$ and $\hat{X}_{\uni}'$ we have the action of $i$ is given by the natural action on the first factor. We will denote by $\tilde{Z}_{\uni}' \subset \tilde{X}_{\uni}'$ the subvariety given by $\{(g,h\bQ) \in \bG_{\uni} \times \bG/\bQ \mid h^{-1}gh \in \Sigma\bU_{\bQ}\}$ and by $\tilde{\imath} : \tilde{Z}_{\uni}' \hookrightarrow \tilde{X}_{\uni}'$ the inclusion map. It is clear that we have an isomorphism $\psi : i^*(\mathscr{E}_0\boxtimes\mathscr{L}) \to i^*(\mathscr{E}_0\boxtimes\Ql)$ of local systems and as $\pi\circ\tilde{\imath}\circ j$ (c.f.\ \cref{pa:K-nu-L-ind-A-nu-L}) commutes with $i$ we have an induced isomorphism
\begin{equation*}
\psi' = \tilde{\imath}^*j^*\pi^*(\psi) : i^*\tilde{\imath}^*\overline{\mathscr{E}_0\boxtimes\mathscr{L}} \to i^*\tilde{\imath}^*\overline{\mathscr{E}_0\boxtimes\Ql},
\end{equation*}
which in turn induces an isomorphism
\begin{equation*}
\psi'': i^*\tilde{\imath}^*D = \IC(\tilde{X}_{\uni}',i^*\tilde{\imath}^*\overline{\mathscr{E}_0\boxtimes\mathscr{L}})[\dim\tilde{X}'] \to \IC(\tilde{X}_{\uni}',i^*\tilde{\imath}^*\overline{\mathscr{E}_0\boxtimes\Ql})[\dim\tilde{X}'] = i^*\tilde{\imath}^*D_0.
\end{equation*}
According to \cite[6.6]{lusztig:1984:intersection-cohomology-complexes} we have $(\tau_!D)|_{\overline{Y}_{\uni}} = \tau_!i^*\tilde{\imath}^*D$ and $(\tau_!D_0)|_{\overline{Y}_{\uni}} = \tau_!i^*\tilde{\imath}^*D_0$ hence $\tau_!(\psi'')$ gives an isomorphism $\delta_{\bQ} : \ind_{\bL_w \subseteq \bQ}^{\bG}(A_{\mathscr{L}})|_{\bG_{\uni}} \to \ind_{\bL_w \subseteq \bQ}^{\bG}(A_0)|_{\bG_{\uni}}$ in $\mathscr{D}\bG_{\uni}$ (c.f.\ \cite[\S2.6(c)]{lusztig:1986:on-the-character-values}). Note that \cref{pa:direct-image-commutative} does not apply here because $i$ is a closed immersion.
\end{pa}

\begin{pa}
For each pair $(\bL_w,\bQ) \in \mathcal{Z}$ we will denote by
\begin{equation*}
\lambda_{\mathscr{L},\bQ} : \ind_{\bL_w \subseteq \bQ}^{\bG}(A_{\mathscr{L}})|_{\bG_{\uni}}\to K_w^{\mathscr{L}}|_{\bG_{\uni}} \qquad\text{and}\qquad \lambda_{\bQ} : \ind_{\bL_w \subseteq \bQ}^{\bG}(A_0)|_{\bG_{\uni}} \to K_w|_{\bG_{\uni}}
\end{equation*}
the restrictions to $\bG_{\uni}$ of the canonical isomorphisms described by \cref{prop:K-nu-L-iso-to-induction}. Putting these isomorphisms together with $\delta_{\bQ}$ we can now define an isomorphism
\begin{equation*}
\epsilon := \lambda_{\bQ}\circ\delta_{\bQ}\circ\lambda_{\mathscr{L},\bQ}^{-1} : K_w^{\mathscr{L}}|_{\bG_{\uni}} \to K_w|_{\bG_{\uni}},
\end{equation*}
(c.f.\ \cite[\S2.6(a)]{lusztig:1986:on-the-character-values}). We claim that $\epsilon$ does not depend upon the choice of parabolic subgroup $\bQ$ used to define it. Firstly, from the construction above, it is clear that $\delta_{\bQ}$ does not depend upon the choice of $\bQ$ hence we need only show that the same is true of $\lambda_{\bQ}$ and $\lambda_{\mathscr{L},\bQ}$. However, using the construction in \cref{pa:K-nu-L-ind-A-nu-L} we see that this follows from the commutative diagram
\begin{equation*}
\begin{tikzcd}
\tilde{Y} \arrow{d}[swap]{\ID}\arrow{r}{\kappa_{\bQ}} & \tau_{\bQ}^{-1}(Y) \arrow{d}{\mu}\\
\tilde{Y} \arrow{r}{\kappa_{\bR}} & \tau_{\bR}^{-1}(Y)
\end{tikzcd}
\end{equation*}
where $(\bL_w,\bR) \in \mathcal{Z}$ and $\mu(g,h\bQ) = (g,h\bR)$. This statement is implicitly used in \cite[\S2.6]{lusztig:1986:on-the-character-values}.
\end{pa}

\begin{pa}
Let $v \in W_{\bG}(\bL_w,\mathscr{L}_w)$ then we denote by $\dot{v} \in N_{\bG}(\bL_w,\mathscr{L}_w)$ a representative of $v$. Following \cite[\S2.6(d)]{lusztig:1986:on-the-character-values} we associate to each $v \in W_{\bG}(\bL_w,\mathscr{L}_w)$ isomorphisms $\tilde{\theta}_{\mathscr{L},v} : \ind_{\bL_w \subseteq {}^{\dot{v}}\bP_w}^{\bG}(A_{\mathscr{L}}) \to \ind_{\bL_w \subseteq \bP_w}^{\bG}(A_{\mathscr{L}})$ and $\tilde{\theta}_v : \ind_{\bL_w \subseteq {}^{\dot{v}}\bP_w}^{\bG}(A_0) \to \ind_{\bL_w \subseteq \bP_w}^{\bG}(A_0)$ in the following way. Recall the notational conventions introduced in the proof of \cref{lem:F-action-ind}. Denote by $\varphi_v : \tilde{X}_{\bL_w\subset{}^{\dot{v}}\bP_w}^{\bG} \to \tilde{X}_{\bL_w\subset\bP_w}^{\bG}$ the isomorphism given by $\varphi_v(g,h{}^{\dot{v}}\bP_w) = (g,h\dot{v}\bP_w)$. By \cref{pa:direct-image-commutative} we have $(\tau_{\bL_w\subset{}^{\dot{v}}\bP_w}^{\bG})_!\varphi_v^* = \ID^*(\tau_{\bL_w\subset\bP_w}^{\bG})_! = (\tau_{\bL_w\subset\bP_w}^{\bG})_!$ so we take $\tilde{\theta}_{\mathscr{L},v}$ and $\tilde{\theta}_v$ to be the isomorphisms induced by $\varphi_v$. Using the definition of $\theta_v$ one can check that we have a commutative diagram
\begin{equation*}
\begin{tikzcd}
K_w^{\mathscr{L}}|_{\bG_{\uni}} \arrow{d}{\theta_v}\arrow{r}{\lambda_{\mathscr{L},v}^{-1}} &
\ind_{\bL_w\subseteq{}^{\dot{v}}\bP_w}^{\bG}(A_{\mathscr{L}})|_{\bG_{\uni}} \arrow{d}{\tilde{\theta}_{\mathscr{L},v}}\arrow{r}{\delta_v} & \ind_{\bL_w\subseteq{}^{\dot{v}}\bP_w}^{\bG}(A_0)|_{\bG_{\uni}} \arrow{d}{\tilde{\theta}_v}\arrow{r}{\lambda_v} & K_w|_{\bG_{\uni}} \arrow{d}{\theta_v}\\
K_w^{\mathscr{L}}|_{\bG_{\uni}} \arrow{r}{\lambda_{\mathscr{L},1}^{-1}} & \ind_{\bL_w\subseteq\bP_w}^{\bG}(A_{\mathscr{L}})|_{\bG_{\uni}}\arrow{r}{\delta_1} & \ind_{\bL_w\subseteq\bP_w}^{\bG}(A_0)|_{\bG_{\uni}}\arrow{r}{\lambda_1} & K_w|_{\bG_{\uni}}
\end{tikzcd}
\end{equation*}
where $\lambda_{\mathscr{L},v} = \lambda_{\mathscr{L},{}^{\dot{v}}\bP_w}$ and $\lambda_v = \lambda_{{}^{\dot{v}}\bP_w}$. In particular this shows that the isomorphism $\epsilon$ is an isomorphism of $W_{\bG}(\bL_w,\mathscr{L}_w)$-modules (recall that $\epsilon$ does not depend upon the choice of parabolic subgroup used to define it).

We now wish to check that the isomorphism $\epsilon$ respects the action of the Frobenius endomorphism. In other words let $\phi_{\mathscr{L}}^w : F^*K_w^{\mathscr{L}} \to K_w^{\mathscr{L}}$ and $\phi_{\Ql}^w : F^*K_w \to K_w$ be the isomorphisms defined in \cref{pa:iso-K-v} then we wish to show that $\epsilon\circ\phi_{\mathscr{L}}^w = \phi_{\Ql}^w\circ F^*\epsilon$. Recall that in \cref{pa:isomorphism-cuspidal-case} we fixed isomorphisms $F^*A_{\mathscr{L}} \to A_{\mathscr{L}}$ and $F^*A_0 \to A_0$ and that by \cref{lem:F-action-ind} these respectively induce isomorphisms $\psi_{\mathscr{L},\bQ} : F^*\ind_{\bL_w \subseteq F(\bQ)}^{\bG}(A_{\mathscr{L}}) \to \ind_{\bL_w \subseteq \bQ}^{\bG}(A_{\mathscr{L}})$ and $\psi_{\bQ} : F^*\ind_{\bL_w \subseteq F(\bQ)}^{\bG}(A_0) \to \ind_{\bL_w \subseteq \bQ}^{\bG}(A_0)$. With this we have a commutative diagram
\begin{equation*}
\begin{tikzcd}[column sep=8ex]
F^*(K_w^{\mathscr{L}}|_{\bG_{\uni}}) \arrow{r}{\phi_{\mathscr{L}}^w}\arrow{d}{F^*\lambda_{\mathscr{L},F(\bQ)}^{-1}} & K_w^{\mathscr{L}}|_{\bG_{\uni}} \arrow{d}{\lambda_{\mathscr{L},\bQ}^{-1}}\\
F^*(\ind_{\bL_w\subseteq F(\bQ)}^{\bG}(A_{\mathscr{L}})|_{\bG_{\uni}}) \arrow{r}{\psi_{\mathscr{L},\bQ}}\arrow{d}{F^*\delta_{F(\bQ)}} & \ind_{\bL_w\subseteq \bQ}^{\bG}(A_{\mathscr{L}})|_{\bG_{\uni}}\arrow{d}{\delta_{\bQ}}\\
F^*(\ind_{\bL_w\subseteq F(\bQ)}^{\bG}(A_0)|_{\bG_{\uni}}) \arrow{r}{\psi_{\bQ}}\arrow{d}{F^*\lambda_{F(\bQ)}} & \ind_{\bL_w\subseteq \bQ}^{\bG}(A_0)|_{\bG_{\uni}}\arrow{d}{\lambda_{\bQ}}\\
F^*(K_w|_{\bG_{\uni}}) \arrow{r}{\phi_{\Ql}^w} & K_w|_{\bG_{\uni}}
\end{tikzcd}
\end{equation*}
hence $\epsilon$ commutes with the isomorphisms $\phi_{\mathscr{L}}^w$ and $\phi_{\Ql}^w$ as desired (note that we have used here that $F(\bG_{\uni}) = \bG_{\uni}$).
\end{pa}

\begin{pa}\label{pa:arbitrary-char-sheaf}
We now arrive at our ultimate isomorphism (c.f.\ \cite[\S2.6(e)]{lusztig:1986:on-the-character-values}). For any simple $W_{\bG}(\bL_w,\mathscr{L}_w)$-module $E$ we will denote by $\widehat{E}$ the induced module $\Ind_{W_{\bG}(\bL_w,\mathscr{L}_w)}^{W_{\bG}(\bL_w)}(E)$. Let $K_{w,\widehat{E}} = (K_w)_{\widehat{E}}$ be as in \cref{eq:K-nu-E-L} then for any such $E$ we can define a $W_{\bG}(\bL_w,\mathscr{L}_w)$-module isomorphism $\mathfrak{X} : K_{w,E}^{\mathscr{L}}|_{\bG_{\uni}} \to K_{w,\hat{E}}|_{\bG_{\uni}}$ by taking the composition
\begin{align*}
K_{w,E}^{\mathscr{L}}|_{\bG_{\uni}} &= \Hom_{W_{\bG}(\bL_w,\mathscr{L}_w)}(E,K_w^{\mathscr{L}}|_{\bG_{\uni}})\\
&\cong \Hom_{W_{\bG}(\bL_w,\mathscr{L}_w)}(E,K_w|_{\bG_{\uni}})\\
&\cong \Hom_{W_{\bG}(\bL_w,\mathscr{L}_w)}(E,\Hom_{W_{\bG}(\bL_w)}(\Ql W_{\bG}(\bL_w),K_w))|_{\bG_{\uni}}\\
&\cong \Hom_{W_{\bG}(\bL_w)}(\Ql W_{\bG}(\bL_w)\otimes E,K_w)|_{\bG_{\uni}}\\
&= K_{w,\widehat{E}}|_{\bG_{\uni}}.
\end{align*}
Here we have used $\epsilon$ and the standard isomorphisms given by (2.6) and (2.19) of \cite{curtis-reiner:1981:methods-vol-I}. Chasing through the isomorphisms we can see that
\begin{equation*}
\mathfrak{X}\circ\phi_{K_{w,E}^{\mathscr{L}}} \circ F^*f = \phi_{\Ql}^w\circ F^*\mathfrak{X} \circ F^*f\circ (1\otimes\psi_E)
\end{equation*}
for all $f \in K_{w,E}^{\mathscr{L}}|_{\bG_{\uni}}$.

Using \cref{eq:trace-comparison-iots} and the fact that the modules $\Ql W_{\bG}(\bL_w)\otimes E_{\sigma}$ and $(\Ql W_{\bG}(\bL_w)\otimes E)_{\sigma}$ are isomorphic as $W_{\bG}(\bL_w)$-modules we see that we have an equality
\begin{equation}\label{eq:chi-A-G-uni}
\chi_{A,\phi_A}|_{G_{\uni}} = \sum_{\iota \in \mathscr{I}(\bL,\nu)^F} \langle \widetilde{E}_{\iota}, \Ind_{W_{\bG}(\bL_w,\mathscr{L}_w).F}^{W_{\bG}(\bL_w).F}(\widetilde{E})\rangle_{W_{\bG}(\bL_w).F}\chi_{K_{\iota},\phi_{\iota}^w}.
\end{equation}
Note that in the above we assume that $\widetilde{E}_{\iota}$ and $\widetilde{E}$ are the restrictions of the extensions to the appropriate coset. In particular we have proved (c.f.\ \cref{pa:conventions-coset-identification}) that
\begin{equation*}
m(A,\iota,\phi_{\iota}^v) = \begin{cases}
\langle \widetilde{E}_{\iota}, \Ind_{W_{\bG}(\bL,\mathscr{L}).Fw^{-1}}^{W_{\bG}(\bL).F}(\widetilde{E})\rangle_{W_{\bG}(\bL).F} &\text{if }\iota \in \mathscr{I}[\bL,\nu]\\
0 &\text{otherwise},
\end{cases}
\end{equation*}
for all $\iota \in \mathcal{N}_{\bG}^F$. Note that, as in \cref{eq:bijection-irr-W_G(L)}, we identify the sets of irreducible characters $\Irr(W_{\bG}(\bL))^F$ and $\Irr(W_{\bG}(\bL_w))^F$. Combining this with the statements of \cref{pa:X-and-Y} we obtain our main result.
\end{pa}

\begin{thm}\label{thm:A}
We have the following equality of class functions
\begin{equation*}
\chi_{A,\phi_A}|_{G_{\uni}} = \sum_{\iota',\iota \in \mathscr{I}[\bL,\nu]^F}\langle \widetilde{E}_{\iota}, \Ind_{W_{\bG}(\bL,\mathscr{L}).Fw^{-1}}^{W_{\bG}(\bL).F}(\widetilde{E})\rangle_{W_{\bG}(\bL).F}(-1)^{a_{\iota}}q^{(\dim\bG+a_{\iota})/2}P_{\iota',\iota}Y_{\iota'}^w.
\end{equation*}
\end{thm}

\begin{rem}
Note that the equality \cref{eq:chi-A-G-uni} echoes a similar equality which is known to hold for almost characters (see for instance \cite[Lemme 6.1]{achar-aubert:2007:supports-unipotents-de-faisceaux}).
\end{rem}

\section{Split Elements and Lusztig's Preferred Extensions}\label{sec:split-elements}
\begin{assumption}
From now until the end of this article we assume that $Z(\bG)$ is connected, $\bG/Z(\bG)$ is simple and $p$ is a good prime for $\bG$.
\end{assumption}

\begin{pa}
To make the formula in \cref{thm:A} computationally explicit we must now show how to compute the functions $Y_{\iota}^w$. In \cite[Remark 5.1]{shoji:1987:green-functions-of-reductive-groups}, assuming $F$ is a split Frobenius endomorphism, Shoji has defined the notion of a \emph{split} unipotent element. However, there is also a notion of split unipotent element when $F$ is not split which we would like to recall here. All the statements of this section are due to the combined efforts of Hotta--Springer, Beynon--Spaltenstein and Shoji. To give the definition we must first prepare some preliminary notions and notation.
\end{pa}

\begin{pa}
Given a unipotent element $u \in \bG$ we denote by $\mathfrak{B}_u^{\bG}$ the variety of Borel subgroups of $\bG$ containing the unipotent element $u$ and we denote its dimension $\dim\mathfrak{B}_u^{\bG}$ by $d_u$. The corresponding $\ell$-adic cohomology groups with compact support $H_c^i(\mathfrak{B}_u^{\bG}) := H_c^i(\mathfrak{B}_u^{\bG},\Ql)$ are modules for the direct product $A_{\bG}(u)\times W_{\bG}$ which are zero unless $i \in \{0,\dots,2d_u\}$ (in fact $i$ must be even). Note that we take the $(A_{\bG}(u)\times W_{\bG})$-module structure to be the one described by Lusztig in \cite[\S3]{lusztig:1981:green-polynomials-and-singularities} (this agrees with the generalised Springer correspondence given in \cite{lusztig:1984:intersection-cohomology-complexes}). This differs from the original module structure given by Springer in \cite{springer:1978:construction-of-representations-of-weyl-groups} but one is translated to the other by composing the $W_{\bG}$-action with the sign character (see \cite[Theorem 1]{hotta:1981:on-springers-representations}).

Assume now that $u \in G$ is fixed by $F$ then $F$ stabilises $\mathfrak{B}_u^{\bG}$ hence we have an induced action of $F$ in the compactly supported cohomology which we denote by $F^{\bullet} : H_c^{\bullet}(\mathfrak{B}_u^{\bG}) \to H_c^{\bullet}(\mathfrak{B}_u^{\bG})$. Let $\psi \in \Irr(A_{\bG}(u))$ be an irreducible character then we denote by $H_c^{\bullet}(\mathfrak{B}_u^{\bG})_{\psi}$ the $\psi$-isotypic component of the module $H_c^{\bullet}(\mathfrak{B}_u^{\bG})$. Springer's main result (see \cite[Theorem 1.13]{springer:1978:construction-of-representations-of-weyl-groups}) shows that either $H_c^{2d_u}(\mathfrak{B}_u^{\bG})_{\psi}$ is 0 or a simple $(A_{\bG}(u) \times W_{\bG})$-module isomorphic to $\psi\otimes E_{u,\psi}$ (this is the classical formulation of the Springer correspondence).
\end{pa}

\begin{pa}
Let $\sigma : W_{\bG} \to W_{\bG}$ denote the automorphism induced by $F$; note that $\sigma$ stabilises the set of Coxeter generators $\mathbb{S}$ so that $\sigma$ induces an automorphism of $(W_{\bG},\mathbb{S})$ (c.f.\ \cref{pa:grothendieck-group}). We now define what it means for a unipotent element $u \in G$ to be \emph{split}. This definition will be well-defined up to $G$-conjugacy. We will do this on a case by case basis as follows.
\begin{enumerate}[label=(\roman*)]
	\item $\bG$ not of type $\E_8$ and $\sigma$ the identity. We say $u \in \bG^F$ is split if $F$ stabilises every irreducible component of $\mathfrak{B}_u^{\bG}$. \cite{beynon-spaltenstein:1984:green-functions,hotta-springer:1977:a-specialization-theorem,shoji:1982:on-the-green-polynomials-F4,shoji:1983:green-polynomials-of-classical-groups}\label{it:adjoint-simple}
	\item $\bG$ of type $\E_8$. If $q \equiv 1 \pmod{3}$ then we define split elements as in case (1). If $q \equiv -1 \pmod{3}$ then for each $F$-stable class $\mathcal{O}$ we define an element $u \in\mathcal{O}^F$ to be split if it satisfies the condition of (1) unless $\mathcal{O}$ is the class $\E_8(b_6)$ in the Bala--Carter labelling (see \cite[pg.\ 177]{carter:1993:finite-groups-of-lie-type}). For this class we say $u \in \mathcal{O}^F$ is split if the restriction of $F^{\bullet}$ to $H^{2d_u}(\mathfrak{B}_u)_{\psi}$ is $q^{d_u}\cdot\ID$ for each irreducible character $\psi$ of $A_{\bG}(u) \cong \mathfrak{S}_3$ which is not the sign character. \cite[\S3 - Case (V)]{beynon-spaltenstein:1984:green-functions}\label{it:adjoint-simple-E8}
	\item $\bG$ of type $\A_n$ and $\sigma$ of order 2. We define any unipotent element $u \in \bG^F$ to be split.\label{it:adjoint-simple-An}
	\item $\bG$ of type $\D_n$ and $\sigma$ of order 2. Let $\bG_{\ad}$ denote $\SO_{2n}(\mathbb{K})/Z(\SO_{2n}(\mathbb{K}))$, which is an adjoint group of type $\D_n$. Let $\pi : \SO_{2n}(\mathbb{K}) \to \bG_{\ad}$ be the natural projection map and let us fix an adjoint quotient $\alpha : \bG \to \bG_{\ad}$. We assume that $\SO_{2n}(\mathbb{K})$ and $\bG_{\ad}$ are endowed with Frobenius endomorphisms so that the morphisms $\pi$ and $\alpha$ are defined over $\mathbb{F}_q$. We say $u \in G$ is split if $(\pi^{-1}\circ\alpha)(u)$ is a split unipotent element as defined in \cite[2.10.1]{shoji:2007:generalized-green-functions-II}. Note that there can only be one unipotent element in the preimage of $\alpha(u)$ under $\pi$.\label{it:adjoint-simple-Dn}
	\item $\bG$ of type $\D_4$ and $\sigma$ of order 3. We say $u \in \bG^F$ is split if it satisfies the condition of (1) with respect to the Frobenius endomorphism $F^3$. \cite[\S4.24]{shoji:1983:green-polynomials-of-classical-groups}\label{it:adjoint-simple-D4}
	\item $\bG$ of type $\E_6$ and $\sigma$ of order 2. Let $F_0 : \bG \to \bG$ be a Frobenius endomorphism which induces the identity on $W_{\bG}$ and is such that $F_0^2 = F^2$. In \cite[\S4]{beynon-spaltenstein:1984:green-functions} Beynon--Spaltenstein define a bijection between the set of $\bG^F$-conjugacy classes of unipotent elements and the set of $\bG^{F_0}$-conjugacy classes of unipotent elements, which is uniquely characterised by three properties. We say a unipotent element $u \in \bG^F$ is split if it is contained in a $\bG^F$-conjugacy class which is in bijective correspondence with a split $\bG^{F_0}$-conjugacy class under Beynon--Spaltenstein's bijection.\label{it:adjoint-simple-E6}
\end{enumerate}

\begin{rem}
In \cref{it:adjoint-simple-Dn} we have used an adjoint quotient $\alpha : \bG \to \bG_{\ad}$ to define the notion of split unipotent element. We claim that this does not depend upon the choice of $\alpha$ up to $G$-conjugacy. Assume $\alpha'$ is another adjoint quotient of $\bG$ defined over $\mathbb{F}_q$ then there exists $t \in \bT_0$ such that $\alpha' = \alpha \circ \ad t$ (see \cite[1.5]{steinberg:1999:the-isomorphism-and-isogeny-theorems}). To prove the claim it suffices to show that we may take $t \in \bT_0^F$. Now $\alpha'$, $\alpha$, $\ad t$ and $F$ are bijections on $\bG_{\uni}$ (the set of unipotent elements), hence the condition $F\circ\alpha' = \alpha'\circ F$ implies $\ad(t^{-1}F(t))$ is the identity on $\bG_{\uni}$. This implies $t^{-1}F(t)$ centralises a regular unipotent element so by \cite[Lemma 14.15]{digne-michel:1991:representations-of-finite-groups-of-lie-type} we must have $t^{-1}F(t) \in Z(\bG)$ as it is semisimple. By the Lang--Steinberg theorem (applied inside the connected group $Z(\bG)$) there exists $z \in Z(\bG)$ such that $t^{-1}F(t) = z^{-1}F(z)$ so $F(tz^{-1}) = tz^{-1}$. Thus we have $\alpha' = \alpha'\circ \ad z^{-1} = \alpha\circ \ad tz^{-1}$ so we are done.
\end{rem}

As the irreducible components of $\mathfrak{B}_u^{\bG}$ have the same dimension (see \cite[I, Proposition 1.12]{spaltenstein:1982:classes-unipotentes}) they form a basis for $H_c^{2d_u}(\mathfrak{B}_u^{\bG})$. In particular \cref{it:adjoint-simple} is equivalent to saying that $F^{\bullet}$ acts as $q^{d_u}$ times the identity on $H_c^{2d_u}(\mathfrak{B}_u^{\bG})$. With this we see that the references given above show that every $F$-stable unipotent conjugacy class $\mathcal{O} \subset \bG$ contains a split unipotent element. Before continuing we recall the following properties of split elements.
\end{pa}

\begin{lem}
Assume $\mathcal{O}$ is an $F$-stable unipotent conjugacy class of $\bG$ then the split elements contained in $\mathcal{O}^F$ form a single $G$-conjugacy class.
\end{lem}

\begin{proof}
Assume $\bG$ is of type $\A_n$ then this is trivial as $A_{\bG}(u)$ is trivial. If $\bG$ is of type $\B_n$, $\C_n$ or $\D_n$ then this is clear from the definition (see \cite[\S2.7, \S2.10]{shoji:2007:generalized-green-functions-II}). If $\bG$ is of exceptional type then this is noted by Benyon--Spaltenstein in \cite[\S3]{beynon-spaltenstein:1984:green-functions}. The only case not explicitly dealt with is \cref{it:adjoint-simple-D4}, however this is easily checked.
\end{proof}

\begin{lem}\label{lem:Frob-action}
If $u \in G$ is a split unipotent element then $F$ acts trivially on $A_{\bG}(u)$.
\end{lem}

\begin{proof}
If $\bG$ is of classical type then the action of $F$ on $A_{\bG}(u)$ is always trivial (see for instance the proof of \cite[Proposition 2.4]{taylor:2013:on-unipotent-supports}). Assume now that $\bG$ is of exceptional type then from \cite{beynon-spaltenstein:1984:green-functions} one sees that Beynon--Spaltenstein start with an element satisfying this property then show that it is split, hence this certainly holds.
\end{proof}

\begin{pa}
Let us denote by $\widetilde{W}_{\bG}$ the semidirect product $W_{\bG} \rtimes \langle \sigma \rangle$ where $\langle \sigma\rangle \leqslant \Aut(W_{\bG})$ is the cyclic group generated by $\sigma$. Assume $u \in \bG^F$ is a split unipotent element and $\psi \in \Irr(A_{\bG}(u))$ then we consider the cohomology group $E = H_c^{2d_u}(\mathfrak{B}_u^{\bG})_{\psi}$ as a $W_{\bG}$-module. If $E$ is $\sigma$-stable then (by \cref{lem:Frob-action}) we may consider $H_c^{2d_u}(\mathfrak{B}_u^{\bG})_{\psi}$ as a $\widetilde{W}_{\bG}$-module which we denote by $\widetilde{E}$ (see \cite[\S4.1]{shoji:2007:generalized-green-functions-II} and \cite[3.9, 3.10]{shoji:1983:green-polynomials-of-classical-groups}). The Frobenius then acts as $q^{d_u}\sigma$ on $H_c^{2d_u}(\mathfrak{B}_u^{\bG})_{\psi}$. In particular $\widetilde{E}$ is an extension of $E$ and it is our purpose to now describe this extension (for which we recall the notion of preferred extension introduced in \cref{rem:preferred-extension}).
\end{pa}

\begin{prop}\label{prop:frob-action}
Assume $u \in \bG_{\uni}^F$ is a split element then the action of $\pm\sigma$ on $\widetilde{E}$ makes this the preferred extension of $E$. The sign is always positive unless $\bG$ is of type $\E_8$, $q \equiv -1\pmod{3}$ and $u$ is contained in the class $\E_8(b_6)$.
\end{prop}

\begin{proof}
The case of type $\A_n$ (resp.\ $\E_6$) follows from \cite[Lemma 3.2]{hotta-springer:1977:a-specialization-theorem} (resp.\ \cite[\S4]{beynon-spaltenstein:1984:green-functions}) together with \cite[Proposition 2.23]{geck-malle:2000:existence-of-a-unipotent-support}. Note that the Springer correspondence we have described in \cref{eq:gen-spring-cor} differs from that in \cite{hotta-springer:1977:a-specialization-theorem} by tensoring with the sign character. The cases of type $\D_n$ with $\sigma$ of order $2$ and $\D_4$ with $\sigma$ of order $3$ are given by \cite[Theorem 4.3(ii)]{shoji:2007:generalized-green-functions-II} and \cite[\S4.24]{shoji:1983:green-polynomials-of-classical-groups} respectively. Finally the case of $\E_8$ follows from the discussion in \cite[\S3, Case (V)]{beynon-spaltenstein:1984:green-functions}.
\end{proof}

\section{Computing Functions in Good Characteristic}\label{sec:comp-functions}
\begin{pa}
We will now show how the functions $Y_{\iota}^w$ occuring in \cref{thm:A} can be computed using the notion of split unipotent elements. In particular, we must show how to determine the trace of $\psi_{\iota}^v$ on the stalk $(\mathscr{E}_{\iota})_u$ for a unipotent element $u \in \mathcal{O}_{\iota}^F$. We will start by considering the case where $u$ is a split unipotent element. Note that the isomorphism $\psi_{\iota}^v$ (c.f.\ \cref{pa:a-and-b-values}) is determined once we have chosen the isomorphism $\varphi_0 : F^*\mathscr{E}_0 \to \mathscr{E}_0$ fixed in \cref{pa:isomorphism-cuspidal-case}.

\begin{assumption}
We assume that the element $u_0 \in \mathcal{O}_0^F$ chosen in \cref{sec:bases-of-end-A} is a split unipotent element. Furthermore, we assume that the isomorphism $\varphi_0 : F^*\mathscr{E}_0 \to \mathscr{E}_0$ is chosen such that the induced map $(\mathscr{E}_0)_{u_0} \to (\mathscr{E}_0)_{u_0}$ at the stalk of the split element $u_0$ is $q^{(\dim(\bL/Z^{\circ}(\bL))-\dim\mathcal{O}_0)/2}$ times the identity (note that $\mathcal{O}_0$ cannot be the class $\E_8(b_6)$).
\end{assumption}

\noindent With these assumptions we may now prove our final result, which follows the line of argument given in \cite[3.4]{lusztig:1986:on-the-character-values} and \cite[Lemma 3.6]{shoji:1997:unipotent-characters-of-finite-classical-groups}.
\end{pa}

\begin{prop}\label{prop:F-stable-pair-loc-sys-iso}
If $\bG$ is of type $\E_8$ then assume $q \equiv 1\pmod{3}$. For any $F$-stable pair $\iota \in \mathcal{N}_{\bG}^F$ and $v \in W_{\bG}(\bL_{\iota})$ the map $(\mathscr{E}_{\iota})_u \to (\mathscr{E}_{\iota})_u$ induced by $\phi_{\iota}^v$ for some (any) split element $u \in \mathcal{O}_{\iota}^F$ is $\pm\gamma_{\bL_{\iota},u_0}^{\bG}(F(v))q^{(\dim\bG + a_{\iota})/2}$ times the identity. The sign is positive unless $\bG$ is of type $\E_8$, $q \equiv -1\pmod{3}$ and $\mathcal{O}_{\iota}$ is the class $\E_8(b_6)$ in which case the sign is negative.
\end{prop}

\begin{proof}
Let $K$ be the image of $(\bL,\mathcal{O}_0,\mathscr{E}_0,\Ql)$ under the map in \cref{eq:ind-w-P-construction} and let $\phi : F^*K \to K$ be the isomorphism defined in \cref{pa:iso-K-v}. We denote by $\varphi_u$ the isomorphism $\mathscr{H}^{a_{\iota}}_u(K) \to \mathscr{H}^{a_{\iota}}_u(K)$ induced by $\Theta_{F(v)}'\phi$. By \cref{prop:bases-of-A} it suffices to show that $\varphi_u$ is $\pm q^{(\dim\bG+a_{\iota})/2}$ times the identity, where the sign is positive unless $\bG$ is of type $\E_8$, $q \equiv -1 \pmod{3}$ and $\mathcal{O}_{\iota}$ is the class $\E_8(b_6)$.

First of all let us note that under the isomorphism in \cref{eq:K-nu-L-decomp} we have
\begin{equation*}
\bigoplus_{E \in \Irr(\mathcal{A})} (E\otimes F^*K_E) \cong F^*K \qquad\qquad \bigoplus_{E \in \Irr(\mathcal{A})} (E \otimes \mathscr{H}_u^{a_{\iota}}(K_E)) \cong \mathscr{H}_u^{a_{\iota}}(K).
\end{equation*}
Both of these isomorphisms are given in the same way as \cref{eq:K-nu-L-decomp} by noticing that $F^*K_E = \Hom_{\mathcal{A}}(E,F^*K)$ and $\mathscr{H}_u^{a_{\iota}}(K_E) = \Hom_{\mathcal{A}}(E,\mathscr{H}_u^{a_{\iota}}(K))$. Assume $E \in \Irr(W_{\bG}(\bL))^F$ then, under the isomorphism in \cref{eq:K-nu-L-decomp}, the restriction of $\Theta_{F(v)}'\phi$ to the summand isomorphic to $E\otimes K_E$ corresponds to the isomorphism $\psi_E^{-1} \otimes (\Theta_{F(v)}'\phi_A)$ (where $\psi_E$ and $\phi_A$ are as in \cref{pa:iso-on-arbitrary-char-sheaf}). With this in mind it is enough to show that $\varphi_u$ acts on an $\mathcal{A}$-submodule of $\mathscr{H}_u^{a_{\iota}}(K)$ isomorphic to $E$ as $q^{(\dim\bG+a_{\iota})/2}$ times $\psi_E^{-1}$ (resp.\ $-\psi_E^{-1}$ if $\mathcal{O}_{\iota}$ is $\E_8(b_6)$ and $q \equiv -1\pmod{3}$).

To prove this we follow the argument of \cite[Lemma 3.6]{shoji:1997:unipotent-characters-of-finite-classical-groups} which in turn is a modification of the arguments in \cite{lusztig:1986:on-the-character-values}. Let $\bP \leqslant \bG$ be a parabolic subgroup of $\bG$ such that $\bP = \bL_{\iota}\bU_{\bP}$ is a Levi decomposition of $\bP$. As $\iota \in \mathcal{N}_{\bG}^F$ is $F$-stable we may assume $\bL_{\iota}$ and $\bP$ are $F$-stable. We define $Z_u$ to be the variety $\{x\bP \in \bG/\bP \mid x^{-1}ux \in \mathcal{O}_0\bU_{\bP}\}$ and similarly we define $\widehat{Z}_u$ to be the variety $\{x \in \bG \mid x^{-1}ux \in \mathcal{O}_0\bU_{\bP}\}$. Now, we have two natural morphisms $\psi_u : \widehat{Z}_u \to Z_u$, resp.\ $\lambda_u : \widehat{Z}_u \to \mathcal{O}_0$, given by $\psi_u(x) = x\bP$ and $\lambda_u(x) = (\mathcal{O}_0$-component of $x^{-1}ux \in \mathcal{O}_0\bU_{\bP})$. Given this we define a local system $\widehat{\mathscr{E}}_{\iota}$ on $Z_u$ by the condition that $\psi_u^*\widehat{\mathscr{E}}_{\iota} = \lambda_u^*\mathscr{E}_0$.

As mentioned in \cref{pa:a-and-b-values}, we have $b_{\iota} = a_{\iota} + \dim\supp K_{\iota}$ hence by \cite[24.2.5]{lusztig:1986:character-sheaves-V} we have an isomorphism $\Phi : \mathscr{H}^{a_{\iota}}_u(K) \to H_c^{b_{\iota}}(Z_u,\widehat{\mathscr{E}}_{\iota})$ (where this is the cohomology with compact support of $Z_u$ with coefficients in the local system $\widehat{\mathscr{E}}_{\iota}$). Note that we have a commutative diagram
\begin{center}
\begin{tikzcd}
Z_u \arrow{d}{\pi_v} & \widehat{Z}_u \arrow{d}{\hat{\pi}_v}\arrow{l}[swap]{\psi_u}\arrow{r}{\lambda_u} & \mathcal{O}_0 \arrow{d}{\ad \dot{v}}\\
Z_u & \widehat{Z}_u \arrow{l}{\psi_u}\arrow{r}[swap]{\lambda_u} & \mathcal{O}_0
\end{tikzcd}
\end{center}
where $\pi_v$ and $\hat{\pi}_v$ are defined by $\pi_v(x) = x\dot{v}^{-1}$ and $\hat{\pi}_v(x\bP) = x\dot{v}^{-1}\bP$. Let $\varphi_0 : F^*\mathscr{E}_0 \to \mathscr{E}_0$ be as in \cref{pa:isomorphism-cuspidal-case} then the composition $\theta_{F(v)}'\varphi_0 : F^*\mathscr{E}_0 \to (\ad\dot{v})^*\mathscr{E}_0$ is an isomorphism such that the induced isomorphism at the stalk of our fixed split element $u_0 \in \mathcal{O}_0^F$ is the identity (c.f.\ \cref{pa:Bonnafe-iso-A}). This isomorphism induces an isomorphism $\widehat{\varphi}_0 : F^*\widehat{\mathscr{E}}_{\iota} \to \pi_v^*\widehat{\mathscr{E}}_{\iota}$ which in turn induces a linear map $\widehat{\varphi}_0$ of $H_c^{b_{\iota}}(Z_u,\widehat{\mathscr{E}}_{\iota})$ satisfying $\Phi\circ\varphi_u = \widehat{\varphi}_0\circ\Phi$. With this we see that we are left with showing that $\widehat{\varphi}_0$ acts on an $\mathcal{A}$-submodule of $H_c^{b_{\iota}}(Z_u,\widehat{\mathscr{E}}_{\iota})$ isomorphic to $E$ as $q^{(\dim\bG+a_{\iota})/2}$ times $\psi_E^{-1}$ (resp.\ $-\psi_E^{-1}$ if $\mathcal{O}_{\iota}$ is $\E_8(b_6)$ and $q \equiv -1\pmod{3}$).

Assume $\bL_{\iota}$ is a torus so that $\iota$ is in the Springer block then $(\mathcal{O}_0,\mathscr{E}_0) = (\{1\},\Ql)$ and $Z_u$ can be canonically identified with $\mathfrak{B}_u^{\bG}$. Furthermore we have $\widehat{\mathscr{E}}_{\iota}$ is simply the constant sheaf and $\widehat{\varphi}_0$ induces the identity at every stalk of $\widehat{\mathscr{E}}_{\iota}$. Hence the statement is simply that of \cref{prop:frob-action} so we are done in this case. The case where $\bL_{\iota} = \bG$ (i.e.\ $\iota$ is cuspidal) is trivial so we are left only with the case where $\bL_{\iota}$ is neither $\bG$ nor $\bT_0$. For this to be the case we must have $\bG$ is of type $\B_n$, $\C_n$ or $\D_n$ but these cases are dealt with by Shoji in \cite[Theorem 4.3]{shoji:2007:generalized-green-functions-II} (see also the reduction arguments given in \cite[1.5]{shoji:2006:generalized-green-functions-I}). Note that to apply this theorem we need the fact that $\theta_{F(v)}'\varphi_0$ induces the identity at the stalk of the split element $u_0 \in \mathcal{O}_0^F$.
\end{proof}

\begin{pa}
With this we may now give a precise description of the function $Y_{\iota}^w$ for all $\iota \in \mathscr{I}[\bL,\nu]^F$. Let us fix a split element $u \in \mathcal{O}_{\iota}^F$ then we denote by $A_{\bG}(u)$ the finite component group $C_{\bG}(u)/C_{\bG}^{\circ}(u)$ of the centraliser of $u$. Recall that we have a natural bijection between $\bG$-equivariant local systems on $\mathcal{O}_{\iota}$ and $\Irr(A_{\bG}(u))$ (see \cite[pg.\ 74]{shoji:1988:geometry-of-orbits}). In particular $\mathscr{E}_{\iota}$ determines a unique irreducible character $\chi_{\iota} \in \Irr(A_{\bG}(u))$. As $F$ acts trivially on $A_{\bG}(u)$ (c.f.\ \cref{lem:Frob-action}) we have the orbits of $G$ acting on $\mathcal{O}_{\iota}^F$ by conjugation are in bijection with the conjugacy classes of $A_{\bG}(u)$. For each $a \in A_{\bG}(u)$ we denote by $u_a \in \mathcal{O}_{\iota}^F$ an element corresponding to $a$ under this bijection. Now, combining \cref{lem:Frob-action,prop:F-stable-pair-loc-sys-iso} and \cite[1.3]{shoji:2006:generalized-green-functions-I} we obtain our final result.
\end{pa}

\begin{thm}\label{thm:Y-functions}
For any $g \in G$ we have
\begin{equation*}
Y_{\iota}^w(g) = \begin{cases}
\pm \gamma_{\bL_{\iota},u_0}^{\bG}(F(w))\chi_{\iota}(a) &\text{if }g\sim_G u_a \in \mathcal{O}_{\iota}^F,\\
0 &\text{otherise},
\end{cases}
\end{equation*}
where the sign is positive unless $\bG$ is of type $\E_8$, $q \equiv -1 \pmod{3}$ and $\mathcal{O}_{\iota}$ is the class $\E_8(b_6)$. Note that we write $g\sim_G u_a$ to denote that $g$ is $G$-conjugate to $u_a$.
\end{thm}

\appendix
\section{Finite Groups and Cosets of Automorphisms}\label{sec:cosets}
\begin{pa}\label{pa:autos-of-fin-groups}
Assume $G$ is a finite group and $\phi : G \to G$ is an automorphism then we recall the notation of \cref{pa:semidirect-products} (in particular $\widetilde{G} = G \rtimes \langle \phi \rangle$). If $V$ is a $G$-module then we denote by $V_{\phi}$ the module obtained from $V$ by twisting with $\phi$. In other words, if $\cdot$ denotes the original action of $\Ql G$ on $V$ then we obtain $V_{\phi}$ by defining a new action $\star$ given by $g \star v = \phi(g)\cdot v$ for all $g \in \Ql G$ and $v \in V$ (here $\phi$ is naturally extended to $\Ql G$ by linearity). We say $V$ is $\phi$-stable if we have an isomorphism $\sigma : V \to V_{\phi}$ of $G$-modules. Note that the choice of $\sigma$ defines a $\widetilde{G}$-module structure on $V$ by setting $\phi\cdot v = \sigma(v)$ for all $v \in V$ and all such possible extensions are obtained in this way. Assume $\chi$ (resp.\ $\widetilde{\chi}$) is the character of $G$ (resp.\ $\widetilde{G}$) afforded by $V$ then we call $\widetilde{\chi}$ the \emph{$\sigma$-extension of $\chi$}.

For any two elements $x,y \in G$ we say $x$ and $y$ are $\phi$-conjugate, and write $x \sim_{\phi} y$, if there exists $z \in G$ such that $x = z^{-1}y\phi(z)$. We denote by $\Cl_{G,\phi}(x)$ the equivalence class of $x$ under $\sim_{\phi}$, called the $\phi$-conjugacy class of $x$, and $H^1(G,\phi)$ the set of all such equivalence classes. Finally we write $C_{G,\phi}(x) = \{g \in G \mid g^{-1}x\phi(g) = x\}$ for the $\phi$-centraliser of $x \in G$; clearly we have $|G| = |C_{G,\phi}(x)|\cdot|\Cl_{G,\phi}(x)|$ by the orbit-stabiliser theorem. When $\phi = \ID$ then we write $\Cl_G(x)$ (resp.\ $C_G(x)$) for $\Cl_{G,\phi}(x)$ (resp.\ $C_{G,\phi}(x)$). We recall here the following lemma which is left as an easy exercise for the reader.
\end{pa}

\begin{lem}\label{lem:tilde-conj-classes}
The natural bijection $G \to G.\phi^i$ restricts to an injective map $\Cl_{G,\phi^i}(x) \to \Cl_{\widetilde{G}}(x)$. Furthermore, when $i=1$ this map is a bijection.
\end{lem}

\subsection{Lusztig--Macdonald--Spaltenstein Induction and Cyclic Automorphisms}
\begin{pa}\label{pa:molien-series}
Let $t$ be an indeterminate over $\Ql$ and let $V$ be a $G$-module. For any irreducible character $\chi \in \Irr(G)$ we will denote by $P_{\chi}^V(t)$ the \emph{Molien series} of $\chi$ (see \cite[\S 5.2.2]{geck-pfeiffer:2000:characters-of-finite-coxeter-groups} for a definition). There then exist unique integers $\gamma_{\chi}^V \in \mathbb{N}$ and $b_{\chi}^V \in \mathbb{N}_0$ such that
\begin{equation*}
P_{\chi}^V(t) = \gamma_{\chi}^V \cdot t^{b_{\chi}^V} +\text{higher powers of }t.
\end{equation*}
We call $b_{\chi}^V$ the \emph{$b$-invariant} of $\chi$. For any natural number $d \in \mathbb{N}_0$ we will denote by $\Irr^V(G\mid d)$ the set $\{\chi \in \Irr(G) \mid \gamma_{\chi}^V = 1$ and $b_{\chi}^V = d \}$.

We will denote by $U$ the $H$-module $V/\Fix_H(V)$ where $\Fix_H(V) = \{v \in V \mid h\cdot v = v$ for all $h \in H\}$. Following \cite[\S 5.2.8]{geck-pfeiffer:2000:characters-of-finite-coxeter-groups} we may define a map $j_H^G : \Irr^U(H \mid d) \to \Irr^V(G \mid d)$, called the \emph{$j$-induction}, by setting $j_H^G(\chi)$ to be the unique irreducible character $\psi \in \Irr(G)$ satisfying:
\begin{itemize}
	\item $\langle \Ind_H^G(\chi),\psi\rangle = 1$ and
	\item $b_{\psi}^V = b_{\chi}^U$.
\end{itemize}
(see \cite[Theorem 5.2.6]{geck-pfeiffer:2000:characters-of-finite-coxeter-groups}). Note that for arbitrary finite groups this induction map was introduced by Lusztig--Spaltenstein in \cite{lusztig-spaltenstein:1979:induced-unipotent-classes}. When there is no confusion concerning the module $V$ we will simply suppress the superscript $V$ in the notation.
\end{pa}

\begin{pa}\label{pa:cyclic-permutation}
We now consider the special case where $G = G_1\times\cdots\times G_n$, with all factors isomorphic, and $\phi$ cyclically permutes the factors. In other words $\phi(g_1,\dots,g_n) = (\phi_n(g_n),\phi_1(g_1),\dots,\phi_{n-1}(g_{n-1}))$ where $\phi_i : G_i \to G_{i+1}$, for $1 \leqslant i \leqslant n-1$, and $\phi_n : G_n \to G_1$ are isomorphisms. We then have the homomorphism
\begin{equation*}
\psi_i := \phi_{i-1}\cdots\phi_1\phi_n\cdots\phi_i : G_i \to G_i
\end{equation*}
is an automorphism of $G_i$ satisfying $\phi_i\psi_i = \psi_{i+1}\phi_i$ for all $1 \leqslant i \leqslant n-1$ and $\phi_n\psi_n = \psi_1\phi_n$. Assume $\chi \in \Irr(G)^{\phi}$ is a $\phi$-stable character then we necessarily have $\chi = \chi_1\boxtimes \cdots \boxtimes \chi_n$ where $\chi_i \in \Irr(G_i)^{\psi_i}$ is $\psi_i$-stable and $\chi_i = \chi_{i+1}\circ\phi_i$ for all $1 \leqslant i \leqslant n-1$ and $\chi_n = \chi_1\circ\phi_n$. In particular, we have a natural bijection between the sets $\Irr(G)^{\phi}$ and $\Irr(G_i)^{\psi_i}$. With this in mind we have the following easy lemma.
\end{pa}

\begin{lem}\label{lem:phi-conj-classes}
The natural inclusion map $\overline{\phantom{x}} : G_1 \to G$ (i.e.\ $\overline{g} = (g,1,\dots,1)$) induces a bijection $H^1(G_1,\psi_1) \to H^1(G,\phi)$ such that $|C_{G_1,\psi_1}(g)| = |C_{G,\phi}(\overline{g})|$.
\end{lem}


\begin{pa}\label{pa:module-V}
We now assume that $V = V_1 \oplus \cdots \oplus V_n$ is a $G$-module such that $V_i\neq 0$ is a $G_i$-module and $V = V_i \oplus \Fix_{G_i}(V)$ for all $1 \leqslant i \leqslant n$. Consequently we have $\Fix_G(V) = \Fix_{G_i}(V_i) = 0$ for all $1 \leqslant i \leqslant n$. Furthermore we assume that $V$ is $\phi$-stable and that there exists an isomorphism $\sigma : V \to V_{\phi}$ with the following properties. There exist $\Ql$-vector space isomorphisms $\sigma_i : V_i \to V_{i+1}$, for $1 \leqslant i \leqslant n-1$, and $\sigma_n : V_n \to V_1$ such that
\begin{equation*}
\sigma(v_1,\dots,v_n) = (\sigma_n(v_n),\sigma_1(v_1),\dots,\sigma_{n-1}(v_{n-1}))\qquad\text{and}\qquad\sigma_i(g_i\cdot v_i) = \phi_i(g_i)\cdot\sigma_i(v_i)
\end{equation*}
for all $g_i \in G_i$ and $v_i \in V_i$. Each linear map $\tau_i := \sigma_{i-1}\cdots\sigma_1\sigma_n\cdots\sigma_i$ is then an isomorphism $V_i \to (V_i)_{\psi_i}$ satisfying $\sigma_i\tau_i\sigma_i^{-1} = \tau_{i+1}$ for all $1 \leqslant i \leqslant n-1$ and $\sigma_n\tau_n\sigma_n^{-1} = \tau_1$. The isomorphism $\sigma$ (resp.\ $\tau_i$) then makes $V$ (resp.\ $V_i$) a $\widetilde{G}$-module (resp.\ $\widetilde{G}_i := G\rtimes\langle \psi_i\rangle$-module) as in \cref{pa:autos-of-fin-groups}. Finally we will assume that there exists a $\sigma$-invariant basis $\mathcal{B} = \mathcal{B}_1\oplus\cdots\oplus\mathcal{B}_n$ of $V$, where $\mathcal{B}_i$ is a basis of $V_i$. In particular we have $\sigma_i(\mathcal{B}_i) = \mathcal{B}_{i+1}$ for all $1 \leqslant i \leqslant n-1$ and $\sigma_n(\mathcal{B}_n) = \mathcal{B}_1$ and $\tau_i(\mathcal{B}_i) = \mathcal{B}_i$ for all $1 \leqslant i \leqslant n$.
\end{pa}

\begin{pa}\label{pa:b-preferred-extension}
Let $E = E_1\boxtimes \cdots \boxtimes E_n$ be a $\phi$-invariant simple $G$-module such that the character $\chi$ afforded by $E$ is contained in $\Irr(G\mid d)^{\phi}$. Here $E_i$ denotes a simple $G_i$-module and we write $\chi_i \in \Irr(G_i)^{\psi_i}$ for the character afforded by $E_i$, which must necessarily be contained in $\Irr(G_i \mid d/n)^{\psi_i}$ (see \cite[Exercise 5.2]{geck-pfeiffer:2000:characters-of-finite-coxeter-groups}). As $E$ is $\phi$-invariant there exists a family of $\Ql$-vector space isomorphisms $\delta_i : E_i \to E_{i+1}$, for $1 \leqslant i \leqslant n-1$ and $\delta_n : E_n \to E_1$ such that
\begin{equation*}
\delta_i(g_i\cdot v_i) = \phi_i(g_i)\cdot\delta_i(v_i)
\end{equation*}
for all $1 \leqslant i \leqslant n$. These isomorphisms are uniquely determined up to scalar multiples. As before, each linear map $\alpha_i := \delta_{i-1}\cdots\delta_1\delta_n\cdots\delta_i$ is an isomorphism $E_i \to (E_i)_{\psi_i}$ and we denote by $\widetilde{\chi}_i$ the $\alpha_i$-extension of $\chi_i$. Replacing the $\delta_i$'s by scalar multiples we may, and will, assume that $\widetilde{\chi}_i$ is the unique extension of $\chi_i$ satisfying $b_{\widetilde{\chi}_i} = b_{\chi_i}$ (here the $b$-invariants are taken with respect to the module $V_i$ introduced in \cref{pa:module-V}). In other words $j_{G_i}^{\widetilde{G}_i}(\chi_i) = \widetilde{\chi}_i$. We now define an isomorphism $\alpha : E \to E_{\phi}$ by setting
\begin{equation*}
\alpha(v_1\boxtimes\cdots\boxtimes v_n) = \alpha_n(v_n)\boxtimes\alpha_1(v_1)\boxtimes\cdots\boxtimes\alpha_{n-1}(v_{n-1}).
\end{equation*}
We call the $\alpha$-extension $\widetilde{\chi}$ of $\chi$ the \emph{$b$-preferred extension of $\chi$}. With this we have the following result concerning the $j$-induction.
\end{pa}

\begin{prop}\label{prop:j-induction-b-preferred}
Assume $n$ is a prime number and $\chi \in \Irr(G\mid d)^{\phi}$. If $\widetilde{\chi} \in \Irr(\widetilde{G})$ is the $b$-preferred extension of $\chi$ then we have $j_G^{\widetilde{G}}(\chi) = \widetilde{\chi}$.
\end{prop}

\begin{proof}
It suffices to show that $b_{\widetilde{\chi}}^V = b_{\chi}^V$. For this we recall the following result of Molien (see \cite[Proposition 5.2.4]{geck-pfeiffer:2000:characters-of-finite-coxeter-groups}) which states that
\begin{equation*}
P_{\widetilde{\chi}}^V(t) = \frac{1}{|\widetilde{G}|}\sum_{g \in \widetilde{G}}\frac{\widetilde{\chi}(g^{-1})}{\det_V(\ID_V - t\cdot \varphi_g)}
\end{equation*}
where $\varphi_g : V \to V$ is the map defined by the action of $g$. The term in the sum is constant on $\widetilde{G}$-conjugacy classes. Therefore, using \cref{lem:tilde-conj-classes} and splitting up along cosets we have
\begin{equation}\label{eq:coset-decomp-molien}
P_{\widetilde{\chi}}^V(t) = \frac{1}{[\widetilde{G}:G]}\sum_{\phi^i \in \langle \phi\rangle}\sum_{x \in H^1(G,\phi^i)} \frac{\widetilde{\chi}((x;\phi^i)^{-1})}{|C_{G,\phi^i}(x)|\cdot\det_V(\ID_V - t\cdot \varphi_{(x;\phi^i)})},
\end{equation}
where here the sum is taken over a set of representatives from the equivalence classes $H^1(G,\phi^i)$.

When $i=1$ it follows from \cref{lem:phi-conj-classes} that we need only compute this term on elements of the form $\overline{x}$ where $x \in G_1$. After suitably ordering the basis $\mathcal{B}$ considered in \cref{pa:module-V} (and identifying linear maps with matrices) we have
\begin{equation*}
\varphi_{(\overline{x};1)} = \begin{bmatrix}
A & & & \\
& I_m & &\\
& & \ddots &\\
& & & I_m
\end{bmatrix}
\qquad\text{and}\qquad
\varphi_{(1;\phi)} = \begin{bmatrix}
& B & & \\
& & \ddots &\\
& & & B\\
B & & & 
\end{bmatrix}
\end{equation*}
where $I_m$ denotes the $m\times m$ identity matrix and $m$ is the common dimension of the $V_j$'s. Here $A$ and $B$ are $m\times m$ matrices; in fact $B$ is a permutation matrix. With this we have
\begin{equation*}
\det(\ID_V - t\cdot \varphi_{(\overline{x};\phi)})
=
\det
\begin{bmatrix}
I_m & -t\cdot AB & & & \\
& I_m & -t\cdot B & &\\
& & I_m & \ddots &\\
& & & \ddots & -t\cdot B\\
-t\cdot B & & & & I_m
\end{bmatrix}
=
\det(I_m - t^n\cdot AB^n).
\end{equation*}
Now $\det(I_m - t^n\cdot AB^n)$ is simply $\det_{V_1}(\ID_{V_1} - t^n\cdot \varphi_{(x;\psi_1)})$. Hence, using \cref{lem:phi-conj-classes}, we can deduce that
\begin{equation}\label{eq:coset-sums}
\sum_{x \in H^1(G_1,\psi_1)} \frac{\widetilde{\chi}((\overline{x};\phi)^{-1})}{|C_{G,\phi}(\overline{x})|\cdot\det_V(\ID_V - t\cdot \varphi_{(\overline{x};\phi)})} = \sum_{x \in H^1(G_1,\psi_1)}\frac{\widetilde{\chi}_1((x;\psi_1)^{-1})}{|C_{G_1,\psi_1}(x)| \cdot\det_{V_1}(\ID_{V_1} - t^n\cdot \varphi_{(x;\psi_1)})}
\end{equation}

Let us now consider the remaining cosets for the automorphism $\phi^i$. If $i = nk$ for some $k \in \mathbb{N}$ then we have $\phi^i = \psi_1^k \times\cdots\times\psi_n^k$. If this is not the case then, as $n$ is prime, we have $\phi^i$ is exactly of the form considered in \cref{pa:cyclic-permutation} and the induced automorphism on $G_m$, for any $1 \leqslant m \leqslant n$, is precisely $\psi_m^i$. In particular we may apply the previous argument to this case to obtain that the statement in \cref{eq:coset-sums} holds with $\phi$ replaced by $\phi^i$ and $\psi_1$ replaced by $\psi_1^i$. Putting this together, and using the fact that $\{\psi_1^n,\psi_1^{2n}\dots,\psi_1^{nk}\} = \{\ID,\psi_1,\dots,\psi_1^{k-1}\}$ we get that
\begin{equation*}
P_{\widetilde{\chi}}^V(t) = \frac{1}{n} \left[ (n-1)P_{\widetilde{\chi}_1}^{V_1}(t^n) + P_{\lambda}^{V}(t) \right]
\end{equation*}
where $\lambda$ is the restriction of $\widetilde{\chi}$ to the subgroup $G\rtimes\langle \phi^n\rangle \leqslant \widetilde{G}$. Observe that as $\lambda$ restricted to $G$ is $\chi$ we must have $b_{\lambda}^V \geqslant b_{\chi}^V$. From this it follows that $b_{\widetilde{\chi}}^V = b_{\lambda}^V = b_{\chi}^V$ because $b_{\widetilde{\chi}_1}^{V_1} = b_{\chi_1}^{V_1}$ and the $b$-invariants are compatible with direct products.
\end{proof}

\begin{rem}
The setup here may seem contrived but this situation occurs naturally when $G$ is a Weyl group and $V$ is the natural module. In this case the invariant basis considered in \cref{pa:module-V} is given by a set of simple roots for the underlying root system of $G$.
\end{rem}

\setstretch{0.96}
\renewcommand*{\bibfont}{\small}
\printbibliography
\end{document}